\documentclass[11pt]{amsart}

\usepackage{amssymb,mathrsfs}
\usepackage{fullpage}
\usepackage[all,arc]{xy}
\UseTips

\theoremstyle{plain}
\newtheorem{theorem}{Theorem}[section]
\newtheorem{cor}[theorem]{Corollary}
\newtheorem{lem}[theorem]{Lemma}
\newtheorem{prop}[theorem]{Proposition}

\theoremstyle{definition}
\newtheorem{example}[theorem]{Example}
\newtheorem{defi}[theorem]{Definition}
\newtheorem{rem}[theorem]{Remark}

\numberwithin{equation}{section}
\numberwithin{figure}{section}

\newcommand{\co}{\colon}

\newcommand{\cyc}{\mathrm{c}}
\newcommand{\ground}{\mathbf{k}}
\newcommand{\g}{\mathfrak{g}}
\newcommand{\h}{\mathfrak{h}}
\renewcommand{\k}{\mathfrak{k}}
\renewcommand{\O}{\mathscr{O}}
\renewcommand{\P}{\mathscr{P}}
\newcommand{\B}{\mathsf{B}}
\newcommand{\E}{\operatorname{E}}
\newcommand{\Ecyc}{\mathscr{E}}
\newcommand{\Com}{\mathscr{C}om}
\DeclareMathOperator{\Hom}{Hom}
\DeclareMathOperator{\MC}{MC}
\DeclareMathOperator{\MCmod}{\mathscr{MC}}
\DeclareMathOperator{\Aut}{Aut}
\DeclareMathOperator{\HAut}{HAut}
\DeclareMathOperator{\sAut}{PAut}
\DeclareMathOperator{\im}{Im}
\DeclareMathOperator{\Ker}{Ker}
\DeclareMathOperator{\ad}{ad}
\DeclareMathOperator{\Dev}{D_{ev}}
\DeclareMathOperator{\Dod}{D_{od}}

\DeclareMathOperator{\id}{id}
\DeclareMathOperator{\Der}{Der}
\DeclareMathOperator{\SDer}{SDer}
\DeclareMathOperator{\CE}{CE}
\DeclareMathOperator{\HCE}{HCE}
\DeclareMathOperator{\Tr}{Tr}

\title{Unimodular homotopy algebras and Chern--Simons theory}

\author{C.~Braun}
\address{Mathematisches Institut\\
Universit\"at M\"unster\\
Einsteinstra{\ss}e 62\\48149 M\"unster\\Germany}
\email{c.braun@uni-muenster.de}

\author{A.~Lazarev}
\address{Department of
Mathematics and Statistics\\
Lancaster University\\
Lancaster LA1 4YF\\UK}
\email{a.lazarev@lancaster.ac.uk}

\thanks{This work was partially supported by EPSRC grants EP/J00877X/1 and EP/J008451/1.}
\thanks{The second author is grateful to T.~Voronov for useful discussions.}


\begin{document}

\begin{abstract}
Quantum Chern--Simons invariants of differentiable manifolds are analyzed from the point of view of homological algebra. Given a manifold $M$ and a Lie (or, more generally, an $L_\infty$)~algebra $\g$, the vector space $H^*(M)\otimes\g$ has the structure of an $L_\infty$~algebra whose homotopy type is a homotopy invariant of $M$. We formulate necessary and sufficient conditions for this $L_\infty$~algebra to have a quantum lift. We also obtain structural results on unimodular $L_\infty$~algebras and introduce a doubling construction which links unimodular and cyclic $L_\infty$~algebras.
\end{abstract}

\maketitle
\tableofcontents

\section{Introduction}
\addtocontents{toc}{\protect\setcounter{tocdepth}{1}}
This paper is motivated by the constructions of quantum Chern--Simons invariants, particularly the approach due to Costello \cite{Cos}, which yield potentially new and interesting topological invariants of closed oriented manifolds. In this paper we develop rigorously the aspects of homotopical algebra which are related to these constructions with a view to analysing them from a new perspective, namely that of homological algebra and obstruction theory.

Let $M$ be a closed oriented manifold with a flat bundle $\g$ of Lie super algebras on $M$ with an invariant pairing of parity opposite to that of $\dim M$. Often the dimension of $M$ is odd, or even three, \cite{Wit, kontfd}. When the cohomology $H^*(M,\g)$ is trivial then Axelrod and Singer \cite{AS, AS2} and Kontsevich \cite{kontfd} showed that one can use the perturbative expansion of the Chern--Simons theory to construct topological invariants of the manifold $M$. Costello \cite{Cos} considered the case when the cohomology is non-trivial and showed that the perturbative expansion should, according to the philosophy of the Batalin-Vilkovisky formalism \cite{Sch, Ro}, lead to a function on the cohomology, meaning a formal power series $S(h) = S_0 +S_1h +S_2h^2+\dots \in\hat{S}\Pi H^*(M,\g)[[h]]$, which is a solution to a certain differential equation called the \emph{quantum master equation}. This solution should be well defined up to homotopy of solutions in a certain precise sense. Costello constructed this solution (up to the constant term) for manifolds of any dimension and with $\g$ perhaps even being a bundle of $L_\infty$~algebras. Similar constructions are also given in \cite{Ia}, where a solution is constructed including the constant term, and in \cite{CM} where a finite dimensional model of Chern--Simons theory is analysed. However, we take a very different philosophy to these latter approaches, which tackle the problem with a motivation coming from physics and have different aims to the present paper. We instead employ the language of formal odd symplectic geometry and obstruction theory cf. \cite{Kh, KV1, KV2, HL} for an overview of this language and related constructions from this perspective.

From the point of view of obstruction theory and homological algebra even the existence of this invariant is a somewhat surprising result and we now explain why. The $S_0$ part of the solution $S(h)$ is an odd element of $\hat{S}\Pi H^*(M,\g)$ which satisfies a different equation, the \emph{classical master equation}. This turns out to be precisely the statement that $S_0$ defines a certain algebraic structure, namely a cyclic $L_\infty$~algebra, on the space $H^*(M,\g)$. From this perspective the full solution $S(h)$ can likewise be understood as defining an algebraic structure, namely a quantum $L_\infty$~algebra \cite{Mar}, on the space $H^*(M,\g)$. The cyclic $L_\infty$~algebra defined by $S_0$ can be understood from the viewpoint of algebraic homotopy theory: it is simply the \emph{minimal model} obtained by transferring the cyclic differential graded Lie algebra structure on $\Omega^*(M, \g)$ to its homology via a cyclic generalisation of the homotopy transfer theorem for differential graded Lie algebras; in particular no consideration of Chern--Simons theory is really necessary to construct $S_0$. The connection between formulas for minimal models of $L_\infty$~algebras and Batalin--Vilkovisky type integrals was explicitly spelled out in \cite{Sh}.

The full solution $S(h)$ can therefore be viewed as a `lift' of this cyclic $L_\infty$~algebra to a larger algebraic structure with 'higher genus' products. However, it is \emph{not} the case that any cyclic $L_\infty$~algebra can be considered as a quantum $L_\infty$~algebra with trivial higher genus products, in other words the function $S_0$ will not in general satisfy the quantum master equation and there can be non-vanishing obstructions to lifting. It is in this sense that the existence of $S(h)$ is somewhat surprising: why do these obstructions always vanish for differential graded Lie algebras of the form $H^*(M,\g)$? In fact, we will show that these obstructions actually do \emph{not} always vanish, even for the case when $\g$ is a trivial bundle of Lie algebras, and thus the Chern--Simons invariant \emph{cannot} be constructed without further assumptions on $M$ and $\g$, a fact which has not been addressed in the current literature.

Therefore, one of the main applications of the techniques developed in this paper is to identify necessary and sufficient conditions for these obstructions to vanish (Theorems~\ref{CS2},~\ref{CS1} and~\ref{tenpr}). In particular we will show that if $M$ is simply-connected and has zero Euler characteristic then a quantum lift of $S_0$ always exists. This holds also in the more general case when $\g$ is itself an $L_\infty$~algebra. This can be viewed as an alternative approach to Costello's construction, although we do not address the question whether a \emph{canonical} lift exists. Conversely, if a quantum lift exists, then either $M$ has zero Euler characteristic or $\g$ must satisfy a certain algebraic condition, called \emph{unimodularity}.

While this result is one of the main motivations for developing the theory of cyclic and quantum $L_\infty$~algebras from a homotopical perspective, the theory also yields a number of results of independent interest, particularly concerning the homotopy theory of cyclic algebras and we now highlight some of the more notable aspects of the theory in this introduction.

\subsection*{Cyclic minimal models}
Before considering the full Chern--Simons invariant it is important to first understand the classical part $S_0$ from the perspective of homotopical algebra. The classical homotopy transfer theorem for a differential graded Lie algebra implies that the homology carries the structure of an $L_\infty$~algebra, well defined up to homotopy of $L_\infty$~algebras. However, note that this is \emph{not} sufficient to construct the $S_0$ term: one must give the structure of a \emph{cyclic} $L_\infty$~algebra (which is an $L_\infty$~algebra with an invariant pairing), well defined up to homotopy of \emph{cyclic} $L_\infty$~algebras. The presence of the pairing means that the usual methods of proving the homotopy transfer theorem do not work well here and there does not yet exist a sufficient theory of minimal models for cyclic algebras for our purposes\footnote{This difficulty was also noted by Costello \cite{Cos2} for minimal models of cyclic $A_\infty$~algebras. However, in the context considered in \cite{Cos2} a workaround using the main result of \cite{HL} was still possible.} (in the case of Chern--Simons theory the well-defined nature of $S_0$ can be seem from the geometric nature of the construction, but this is somewhat ad-hoc). The existence of the cyclic minimal model has been shown in \cite{CL} but a stronger result concerning uniqueness up to cyclic homotopy is required.

Therefore, Appendix~\ref{app:operads} proves a version of the minimal model theorem (Theorem~\ref{thm:minmod}) for algebras over cyclic operads. The cyclic case is a good deal more complicated than the non-cyclic one and we develop a completely new approach to minimal models, based on the closed model category structure on cyclic operads. This approach also works, under some restrictions, in the case when the underlying space of an operadic algebra is infinite-dimensional (which is needed in our applications). One particularly noteworthy feature of our treatment is a construction of a cyclic endomorphism operad for a differential graded vector space having \emph{finite-dimensional homology} (Definition~\ref{def:cyclicendomorphism}). In the infinite-dimensional case it is strictly smaller than the usual endomorphism operad; however under certain favorable conditions, usually satisfied in practice, it has the same homotopy type. These results are of independent interest; the main text does not use the language of operads and this appendix can be skipped by the casual reader.

On a related theme, Appendix~\ref{app:tensor} explains how to construct tensor products of homotopy algebras (e.g.~of two $A_\infty$~algebras) in a homotopy invariant way, both in the cyclic and non-cyclic situations. This should also be of independent interest.

\subsection*{Homotopy invariance}
A natural question to ask about any smooth invariant of a manifold is, of course, how coarse an invariant is it? For example, is it a homotopy invariant? Surprisingly, even for the classical $S_0$ term this is non-trivial. While it is clear from the classical homotopy transfer theorem that the non-cyclic $L_\infty$~algebra structure defined by $S_0$ is a homotopy invariant of the manifold (for fixed trivial bundle $\g$) it is less clear that the \emph{cyclic} $L_\infty$~algebra structure is a homotopy invariant of a manifold. Moreover, even after developing the corresponding homotopy transfer theorem for cyclic algebras (Appendix~\ref{app:operads}) it is still not clear due to the extra complications in the cyclic case.

Nevertheless, we are able to settle this question for the classical invariant with the help of the main theorem of \cite{HL} and show that it is indeed a homotopy invariant (Theorem~\ref{cyclicmin}).

However, the nature of whether the full Chern--Simons invariant is a homotopy invariant, or whether it is something stronger, is not yet understood.

\subsection*{The Lambrechts--Stanley theorem}
The Lambrechts--Stanley theorem \cite{LS} proves the existence of a finite dimensional differential graded commutative Frobenius algebra model of the de~Rham algebra of a simply-connected closed orientable manifold and this powerful theorem plays an important role in this paper. It is not claimed, however, that the Lambrechts--Stanley model is unique: it is, obviously, unique up to weak equivalence of commutative differential graded algebras, but not a priori up to equivalence of \emph{cyclic} algebras (and it is not completely clear what this latter notion should mean precisely).

We therefore strengthen the Lambrechts--Stanley theorem in this direction showing that their model is necessarily unique up to strong cyclic homotopy equivalence (Proposition~\ref{LS2}). In particular this means that the cyclic $C_\infty$ minimal model on the homology of the Lambrechts--Stanley Frobenius algebra is homotopy equivalent to the $C_\infty$ minimal model of the de Rham algebra as \emph{cyclic} algebras, not just as $C_\infty$~algebras.

\subsection*{Unimodular $L_\infty$~algebras}
In general it is difficult to decide whether a cyclic $L_\infty$~algebra can be lifted to a quantum $L_\infty$~algebra; the first obstruction is provided by a cohomology class in the \emph{Chevalley--Eilenberg} cohomology of the algebra but even if this vanishes, in general there will be many more further obstructions. We note that the corresponding obstruction theory for lifting cyclic $A_\infty$~algebras to quantum $A_\infty$~algebras was developed in \cite{Ham}; we are considering precisely the $L_\infty$~analogue of this.

We will see that for the special case of $L_\infty$~algebras arising as tensor products of strict differential graded commutative algebras with $L_\infty$~algebras (such as the space $\Omega^*(M,\g)\cong \Omega^*(M)\otimes \g $ where $\g$ is a trivial bundle) it is normally sufficient to consider only lifts to first order. More precisely, we show that this means we can consider \emph{unimodular $L_\infty$~algebras}, introduced in \cite{Gran}, instead of quantum $L_\infty$~algebras and the problem of lifting an $L_\infty$~algebra to a unimodular $L_\infty$~algebra which is a more amenable problem. This is an important fact since unimodularity does not depend on the invariant pairing on a cyclic $L_\infty$~algebra which simplifies the problem further. The notion of unimodularity is already nontrivial for ordinary Lie algebras and it reduces in that case to the the familiar condition of tracelessness of the adjoint representation. 

We therefore analyse this latter lifting problem in some detail (Section~\ref{spaceoflifts}) and give new structural results on unimodular $L_\infty$~algebras, of independent interest. Our main result in this regard is showing that almost every unimodular $L_\infty$~algebra can be reduced to a \emph{strict} unimodular $L_\infty$~algebra (Theorem~\ref{strict}), a result which does not generalise to quantum $L_\infty$~algebras.

We analyse the obstruction theory using the language of \emph{Maurer--Cartan elements} in differential graded Lie algebras. Appendix~\ref{app:mc} recalls basic facts and terminology concerning Maurer--Cartan elements in differential graded Lie algebras and develops the theory of lifts of Maurer--Cartan elements (Theorem~\ref{thm:fibration}).

\subsection*{The doubling construction}
Since the $L_\infty$ minimal model on $H^*(M,\g)$ is not just any $L_\infty$~algebra, but rather a \emph{cyclic} $L_\infty$~algebra, we will want to consider the problem of lifting a \emph{cyclic} $L_\infty$~algebra $V$ to a unimodular $L_\infty$~algebra. It is a straightforward fact that if the invariant pairing on $V$ is of even degree then $V$ always lifts to a unimodular algebra. If the pairing on $V$ is odd then this is less clear.

We therefore introduce the \emph{odd double}, associating to an $L_\infty$~algebra $V$ an odd cyclic $L_\infty$~algebra on a space twice the dimension of $V$ and show that it admits a unimodular (or quantum) lift if and only if the original $L_\infty$~algebra admits a unimodular lift (Theorem~\ref{quantunim}). There is an even analogue of this doubling construction, which we also describe, for completeness. This even analogue was described using the language of operads in \cite{ES} where it was called \emph{cyclic completion}. In contrast, we approach this construction from the viewpoint of formal odd symplectic geometry.

In particular, since there are many examples of non-unimodular Lie algebras this construction gives a way of producing plentiful examples of odd cyclic $L_\infty$~algebras which are not unimodular and thus many examples where the higher dimensional Chern--Simons invariant cannot be constructed.

\addtocontents{toc}{\protect\setcounter{tocdepth}{2}}
\subsection{Notation and conventions}
In this paper we work in the category of $\mathbb Z/2$--graded vector spaces (also known as super vector spaces) over a field $\ground$ of characteristic zero. However, all our results (with obvious modifications) continue to hold in the $\mathbb Z$--graded context. The adjective `differential graded' will mean `differential $\mathbb Z/2$--graded' and will be abbreviated as `dg'. A (commutative) differential graded (Lie) algebra will be abbreviated as (c)dg(l)a. All of our unmarked tensors are understood to be taken over $\ground$. For a super vector space $V=V^0\oplus V^1$ the symbol $\Pi V$ will denote the \emph{parity reversion} of $V$; thus $(\Pi V)^0=V^1$ while $(\Pi V)^1=V^0$.

We will often invoke the notion of a \emph{complete}\footnote{In earlier works of the authors, such as \cite{braun, HL}, the adjective `formal' was used instead of `complete'; we have chosen to use the latter to avoid clashes with established terminology.} (dg) vector space; this is just an inverse limit of finite-dimensional discrete (dg) vector spaces. An example of a complete space is $V^*$, the $\ground$--linear dual to a discrete vector space~$V$. A complete vector space comes equipped with a topology (the inverse limit topology) and whenever we deal with a complete vector space all linear maps from or into it will be assumed to be continuous; thus we will always have $V^{**}\cong V$. If $V$ is a discrete space and $W=\lim_\leftarrow{W_i}$ is a complete space we will write $V\otimes W$ for $\lim_{\leftarrow}V\otimes W_i$; thus for two discrete spaces $V$ and $U$ we have $\Hom(V,U)\cong U\otimes V^*$.

The category of complete dg vector spaces is naturally a symmetric monoidal category with the \emph{completed} tensor product defined for two complete spaces $V=\lim_{\leftarrow} V_i$ and $W=\lim_{\leftarrow} W_j$ as $V\otimes W := \lim_{\leftarrow}V_i\otimes W_j$. Symmetric monoids in this category \emph{which are in addition local} as commutative algebras will be called \emph{complete} cdgas. The free complete commutative algebra on a complete graded space $V$ will be denoted by $\hat{S}V$. The complete algebra $\hat{S}V$ has a grading by the number of symmetric factors; we will write $\hat{S}_{\geq n}V$ for the ideal of $\hat{S}V$ consisting of (possibly infinite) linear combination of monomials in $V$ of length at least $n$. Derivations of the graded algebra $\hat{S}V$ will be denoted by $\Der(\hat{S}V)$; we will write $\Der_{\geq n}(\hat{S}V)$ and $\Der_{n}(\hat{S}V)$ for those derivations which raise the grading by at least $n$ or exactly by $n$, respectively. $\Aut(\hat{S}V)$ stands for the groups of continuous automorphisms of $\hat{S}V$.

\section{Preliminaries on cyclic $L_\infty$~algebras}\label{sec:linfty}
In this section we recall the relevant basic definitions of $L_\infty$~algebras and their cyclic versions. We refer to \cite{HL',CL'} for details.

Let $(V,d)$ be a dg vector space. The differential $d$ on $V$ (which will frequently be omitted from notation) induces the differential $L_d$ on the graded Lie algebra $\Der_{\geq 1}(\hat{S}\Pi V^*)$ acting, as the notation suggests, by the Lie derivative with the linear vector field $d$ and making it a dgla; we will frequently abuse the notation by writing $d$ for $L_d$.
\begin{defi}\
\begin{enumerate}
\item
An $L_\infty$~structure on $V$ is an odd element $m\nobreak\in\nobreak\Der_{\geq 2}(\hat{S}\Pi V^*)$ which satisfies the MC equation $d(m)+\frac{1}{2}[m,m]=\nobreak 0$. The pair $(V,m)$ will be referred to as an $L_\infty$~algebra and the algebra $\hat{S}\Pi V^*$, supplied with the differential $d+m$, as its representing complete cdga.
\item
For an $L_\infty$~algebra $(V,m)$ its representing complete cdga $\hat{S}\Pi V^*$ will often be denoted as $\CE(V)$ and called the \emph{Chevalley--Eilenberg} complex of $V$; note that this is consistent with the usual terminology in the case when $m$ is a quadratic derivation and thus, $V$ is an ordinary graded Lie algebra. We will also write $\CE_{\geq n}(V)$ for the cdga $\hat{S}_{\geq n}\Pi V^*$. The corresponding cohomology will be denoted by $\HCE(V)$ and $\HCE_{\geq n}(V)$ respectively.
\end{enumerate}
\end{defi}

Now suppose that $V$ has a non-degenerate even or odd symmetric bilinear form $\langle,\rangle$; this is equivalent to saying that  $\Pi V$ has a linear even or odd symplectic structure which will be denoted by $\omega$. Of course, we also assume that it is compatible with the differential, which means that for $u,v\in V$ it holds that $\langle d(u),v\rangle+(-1)^{|u|}\langle u,d(v)\rangle=0$.  Then there is a notion of a \emph{cyclic} $L_\infty$~algebra supported on $V$.

\begin{defi}An $L_\infty$~structure $m$ on $V$ is called (even or odd) cyclic if it preserves the given (even or odd) symplectic structure $\omega$ on $\Pi V$, i.e. if $L_m(\omega)=0$ where $L_m$ stands for the Lie derivative along $m$.
\end{defi}

\begin{rem}\
\begin{itemize}
\item
Following tradition, we defined cyclic $L_\infty$~algebras as $L_\infty$~algebras supplied with a \emph{non-degenerate} form. In particular, they must be finite-dimensional. It will be convenient to relax this notion a little by dropping the requirement of non-degeneracy. Suppose that $\langle,\rangle$ is a symmetric bilinear form on a dg vector space $V$. This form can be viewed as a differential two-form $\omega$ on the formal manifold $\Pi V$. Suppose that a vector field $m\in~\Der_{\geq 2}(\hat{S}\Pi V^*)$ preserves $\omega$. If $m$ determines an $L_\infty$~structure on $V$, then we will slightly abuse the terminology by referring to $V$ as a cyclic $L_\infty$~algebra, even though $\langle,\rangle$ might be degenerate. We denote by $\Der_{\geq 2}^\cyc(\hat{S}\Pi V^*)$ the subspace of derivations in $\Der_{\geq 2}(\hat{S}\Pi V^*)$ which preserve $\omega$. This is clearly a dg Lie subalgebra and a more succinct definition of a cyclic $L_\infty$~algebra in this sense is as an MC element in the dgla $\Der_{\geq 2}^\cyc(\hat{S}\Pi V^*)$.
\item
In the definition of an $L_\infty$~algebra, we can let $V$ have vanishing differential and consider instead the dgla $\Der_{\geq 1}(\hat{S}\Pi V^*)$. This results in an equivalent definition; however it is sometimes less convenient since the latter dgla is \emph{not} pronilpotent (it has a reductive part consisting of linear endomorphisms of $\Pi V^*$).
\end{itemize}
\end{rem}

There is a concomitant notion of an $L_\infty$~map.
\begin{defi}\
\begin{enumerate}
\item
Let $((V,d_V),m_V)$ and $((W,d_W),m_W)$ be two $L_\infty$~structures on $V$ and $W$. An $L_\infty$~map $f\co(V,m_V)\to(W,m_W)$
is an algebra map between their representing complete cdgas $f\co\hat{S}\Pi W^*\to \hat{S}\Pi V^*$
such that $f\circ(m_W+d_W)=(m_V+d_V)\circ f$.
\item
If $V$ and $W$ are supplied with bilinear forms (even or odd) and the $L_\infty$~structures $(V,m_V)$ and $(W,m_W)$ are cyclic with respect to these, then an $L_\infty$~map $f\co(V,m_V)\to(W,m_W)$ is \emph{cyclic} if the corresponding map on representing complete cdgas preserves the two-forms (symplectic structures in the non-degenerate case).
\end{enumerate}
\end{defi}
A more traditional approach to defining $L_\infty$~algebras and maps is through multi-linear maps. Note that a derivation $m\in\Der_{\geq 2}\hat{S}\Pi V^*$ has the form $m=m_2+m_3+\dots$ where $m_n$ is the $n$--th homogeneous component of $m$. In other words, any derivation is determined by the collection of maps $m_n\co\Pi V^*\to\left( (\Pi V^*)^{\otimes n}\right )_{S_n}$. We have an identification between $S_n$ coinvariants and $S_n$ invariants:
\[
i_n\co\left(\Pi V^*)^{\otimes n}\right)_{S_n}\to
\left((\Pi V^*)^{\otimes n}\right)^{S_n}\cong
\left((\Pi V^{\otimes n})_{S_n}\right)^*
\]
where $i_n(x_1\otimes \dots \otimes x_n)=\sum_{\sigma\in S_n}\sigma[x_1\otimes\dots\otimes x_n]$. Then the dual to the composite map
\[
i_n\circ m_n\co\Pi V^*\to \left((\Pi V^{\otimes n})_{S_n}\right)^*
\]
is a map $\tilde{m}_n\co(\Pi V^{\otimes n})_{S_n}\to\Pi V$. Thus, an $L_\infty$~structure on $V$ is equivalent to a collection of symmetric multilinear maps $\tilde{m}_n\co (\Pi V)^{\otimes n}\to \Pi V$ of odd degree as above subject to appropriate conditions stemming from the MC equation $d(m)+m\circ m=0$. For uniformity we can set $m_1\co\Pi V^*\to\Pi V^*$ to be the internal differential and then the MC constraint could be written as $M\circ M=0$ where $M:={m}_1+{m}_2+\dots\in  \Der_{\geq 1}(\hat{S}\Pi V^*)$.

A similar argument shows that an $L_\infty$~map $f\co\hat{S}\Pi W^*\to \hat{S}\Pi V^*$ is equivalent to a collection of symmetric multi-linear maps $\tilde{f}_n\co(\Pi V)^{\otimes n}\to \Pi W, n=1,2,\dots$ of even degree satisfying suitable identities. An $L_\infty$~map is \emph{strict} if $\tilde{f}_n=0$ for $n=2,3,\dots$.

From the point of view of multi-linear maps, an $L_\infty$~algebra structure $m$ on a vector space $V$ together with a non-degenerate pairing $\langle,\rangle$ is cyclic if and only if the tensors $\langle \tilde{m}_n(v_1,\dots, v_n), v_{n+1}\rangle$ are $S_{n+1}$--symmetric where $v_1,\dots, v_{n+1}\in V$.

We can now define the notion of an $L_\infty$~(quasi\nobreakdash-)isomorphism.

\begin{defi}
A (cyclic) $L_\infty$~map $f\co(V,m_V)\to (W,m_W)$ is a (cyclic) $L_\infty$~(quasi\nobreakdash-)isomorphism if its linear component $\tilde{f}_1\co \Pi V \to \Pi W$ is a (quasi\nobreakdash-)isomorphism. If the spaces $V$ and $W$ coincide and the component $\tilde{f}_1$ of an $L_\infty$~map $f$ is the identity map, then $f$ is called a \emph{pointed} $L_\infty$~map. A pointed $L_\infty$~map is necessarily an isomorphism.
\end{defi}

\begin{rem}\label{rem:linftygaugeiso}\
\begin{itemize}
\item If one understands an $L_\infty$~structure as an MC element in the pronilpotent dgla $\Der_{\geq 2}(\hat{S}\Pi V^*)$ then the notion of a pointed $L_\infty$~isomorphism becomes identical with the notion of gauge equivalence of MC elements (see Appendix~\ref{app:mc}). When we say two $L_\infty$~algebras are equivalent we will mean that they are gauge equivalent $L_\infty$~structures (and so are supported on the same space, i.e.~pointed $L_\infty$~isomorphic).
\item The same comments apply in the cyclic situation: the notion of a pointed cyclic $L_\infty$~isomorphism is equivalent to the notion of gauge equivalence of MC elements in the pronilpotent dgla $\Der_{\geq 2}^{\cyc}(\hat{S}\Pi V^*)$.
\item Sometimes however, it is natural to consider the action of the reductive part of $\Aut(\hat{S}\Pi V^*)$ on the set of $L_\infty$~structures, even though it does not quite fit with the general theory of MC elements and their gauge transformations.
\end{itemize}
\end{rem}

It is generally more convenient to think about $L_\infty$~algebras geometrically, as formal vector fields of square zero. There is, however, one exception: a tensor product of an $L_\infty$~algebra and a cdga is easier to understand in terms of multi-linear maps.
\begin{defi}
Let $A$ be a cdga and $V$ be an $L_\infty$~algebra specified by a collection of multi-linear maps $\tilde{m}_n^V\co(\Pi V)^{\otimes n}\to V$. Denote by $X^n\co A^{\otimes n}\to A$ the $n$--fold product map. Then $A\otimes V$ has the structure of an $L_\infty$~algebra determined by the multi-linear maps
\[
\tilde{m}^{A\otimes V}_n=X^n\otimes \tilde{m}_n\co(A\otimes V)^{\otimes n}\to A\otimes V.
\]
\end{defi}
Suppose that $A$ has an invariant scalar product $[,]$ so that $[ab,c]=[a,bc]$ for all $a,b,c\in A$ and that $V$ has a scalar product $\langle,\rangle$ making it a cyclic $L_\infty$~algebra; we do not assume finite dimensionality of $A$ or $V$ here. Then the $L_\infty$~algebra $A\otimes V$ can be endowed with a scalar product $(,)$ so that
\[
(a\otimes v,b\otimes u):=(-1)^{|v||b|}[a,b]\langle v, u\rangle
\]
where $a,b\in A$ and $v,u\in V$. The scalar product $(,)$ makes $A\otimes V$ into a cyclic $L_\infty$~algebra.

\section{Lie algebras of vector field and the doubling construction}
In this section we construct a map from the Lie algebra of formal vector fields on a vector space $V$ into the Lie algebra of Hamiltonian vector fields on the space $V^*\oplus V$ as well as an odd analogue of this construction. For future use we need to consider the case when $V$ carries a differential; however all the construction of this section are meaningful and non-trivial in the case when this differential vanishes; in this case various dglas of formal vector fields will also become simply graded Lie algebras.

So let $V$ be a finite-dimensional dg vector space and consider the dgla $\Der(\hat{S} V^*)$ of derivations of the formal functions on $V$, also known as the Lie algebra of formal vector fields on $V$. Any derivation of $\hat{S}V^*$ is uniquely determined by its value on $V^*$ and thus, we have an isomorphism of dg vector spaces
\[
\Der(\hat{S} V^*)\cong \hat{S}V^*\otimes V.
\]
Consider also the dg vector space $V^*\oplus V$; it clearly has a nondegenerate anti-symmetric bilinear form $\langle, \rangle$ which is defined by requiring that $V$ and $V^*$ are both isotropic subspaces, while $\langle v^*,u\rangle=v^*(u)$ where $v^*\in V^*, u\in V$. We will call $V^*\oplus V$ the \emph{even double} of $V$. Similarly, $V^*\oplus \Pi V$ has a nondegenerate symmetric \emph{odd} bilinear form $(,)$ which is defined by requiring that $V$ and $V^*$ are both isotropic subspaces, while $(v^*,\Pi u)=(-1)^{|v^*|}v^*(u)$ where $u\in V$ so that $\Pi u\in \Pi V$ and $v^*\in V^*$. We will call $V^*\oplus \Pi V$ the \emph{odd double} of $V$.

So, the spaces $V\oplus V^*$ and $V^*\oplus \Pi V$ can be regarded as linear symplectic spaces (even and odd respectively) and thus, their formal algebras of functions have associated Poisson brackets. Thus, $\hat{S}(V^*\oplus V)$ becomes a dgla while $\hat{S}(V^*\oplus \Pi V)$ becomes an \emph{odd} dgla; we will refer to them as dglas of formal Hamiltonians. Note that the odd Poisson bracket on $\hat{S}(V^*\oplus \Pi V)$ is sometimes called the \emph{antibracket}. These Lie algebras have a one-dimensional center formed by constant functions and we will denote by $\hat{S}_+(V^*\oplus V)$ and $\hat{S}_+(V^*\oplus \Pi V)$ the corresponding quotient Lie algebras; note that the latter can be identified with the Lie algebras of \emph{symplectic} vector fields on $V^*\oplus V$ and $V^*\oplus \Pi V$ respectively. The associated Lie brackets on $\hat{S}_+(V^*\oplus V)$ and $\hat{S}_+(V^*\oplus \Pi V)$ are the commutators of the corresponding Hamiltonians and will be denoted by $(,)$ and $\{,\}$ respectively. We refer the reader to \cite{HL} for a short survey of the relevant elementary theory.

The bracket $(,)$ makes $\hat{S}_+(V^*\oplus V)$ into a Lie algebra, additionally it satisfies the Leibniz identity: for $f, g, h\in \hat{S}_+(V^*\oplus V)$ we have
\[
(f,gh)=(f,g)h+(-1)^{|f||g|}g(f,h).
\]
The bracket $\{,\}$ makes $\hat{S}_+(V^*\oplus \Pi V)$ into an \emph{odd} Lie algebra; that means that it is symmetric (as opposed to antisymmetric):
\[
\{f,g\}=(-1)^{|f||g|}\{g,f\}
\]
and satisfies the odd Jacobi identity
\[
\{f,\{g,h\}\}=(-1)^{|f|+1}\{\{f,g\},h\}+(-1)^{(|f|+1)(|g|+1)}\{g,\{f,h\}\}.
\]
In addition, the odd Leibniz identity holds:
\[
\{f,gh\}=\{f,g\}h+(-1)^{(|f|+1)|g|}g\{f,h\}
\]
Furthermore, given a basis $x_1,\dots,x_n$ in $V$, the Poisson brackets for linear functions are given as:
\begin{eqnarray}\label{initial}
(x^*_i,x_j)=\{x^*_i,\Pi x_j\}=\delta_{ij}.
\end{eqnarray}

We will now introduce the doubling construction, associating to any formal vector field on $V$ a Hamiltonian on each of the even and odd doubles of $V$. Note that the obvious embedding $V\hookrightarrow \hat{S}V\cong \prod_{k=0}^\infty S^kV$ induces a map
\[
\Dev\co\Der(\hat{S} V^*)\cong \hat{S}V^*\otimes V\hookrightarrow \hat{S} V^*\otimes \hat{S}_+ V\subset\hat{S}_+(V^*\oplus V).
\]
We will also define an \emph{odd} map
\[
\Dod\co\Der(\hat{S} V^*)\cong \hat{S}V^*\otimes V\hookrightarrow \hat{S} V^*\otimes \hat{S}_+\Pi V\subset\hat{S}_+(V^*\oplus \Pi V)
\]
by the formula $\hat{S}V^*\otimes V\ni f\otimes v\mapsto (-1)^{|f|}f\otimes\Pi v\in \hat{S} V^*\otimes \hat{S}_+\Pi V\subset \hat{S}_+(V^*\oplus \Pi V)$.

\begin{defi}The maps $\Dev$ and $\Dod$ are called even and odd doubling maps respectively.
\end{defi}

\begin{theorem}\
\begin{enumerate}
\item The map $\Dev$ is a dgla map, realizing $\Der(\hat{S}V^*)$ as a sub-dgla in the dgla of formal Hamiltonians
on $V^*\oplus V$.
\item
The map $\Dod$ is an \emph{odd} dgla map, i.e.~it satisfies the equality \[\Dod[\xi,\eta]=(-1)^{|\xi|}[\Dod(\xi),\Dod(\eta)].\] This map realizes $\Der(\hat{S}V^*)$ as a sub-dgla in the odd dgla of formal Hamiltonians
on $V^*\oplus \Pi V$.
\end{enumerate}
\end{theorem}
\begin{proof} Note that the maps $\Dev$ and $\Dod$ are clearly injective and respect the differentials. Consider the even case first.
Choose a basis $x_1,\dots, x_n$ in $V^*$ and the dual basis $x_1^*,\dots, x_n^*$ in $V$. Note any derivation of $\hat{S}V^*$ is a linear combination
 of derivations of the form $f\partial_{x_i}$ where $f\in \hat{S}V^*$. Then $\Dev(f\partial_{x_i})=fx_i^*\in \hat{S}(V^*\oplus V)$. It is, therefore, sufficient to prove
 that for any $1\leq i,j\leq n$
 \[
 \Dev[f\partial_{x_i}, g\partial_{x_j}]=(fx_i^*,gx_j^*)\in \hat{S}(V^*\oplus V),
 \]
where $f,g\in \hat{S}V^*$, and $[,], (,)$ stand for the commutator of derivations and the Poisson bracket of formal functions on $V^*\oplus V$ respectively.
We have
\begin{align*}
\Dev[f\partial_{x_i}, g\partial_{x_j}]=&\Dev\left(f\partial_{x_i}(g)\partial_{x_j}+(-1)^{(|f|+|x_i|)|g|}g[f\partial_{x_i},\partial_{x_j}]\right)\\
=&f\partial_{x_i}(g)x_j^*-(-1)^{(|f|+|x_i|)|g|+|x_j|(|f|+|x_i|)}g\partial_{x_j}fx_i^*\\
=&-(-1)^{(|f|+|x_i^*|)|g|}(g,fx_i^*)x^*_j-(-1)^{(|f|+|x_i|)|g|+|x_j|(|f|+|x_i|)}g(x_j^*,fx_i^*)\\
=&(fx_i^*,g)x_j^*+(-1)^{(|f|+|x_i|)(|g|)}g(fx_i^*,x_j^*)\\
=&(fx_i^*,gx_j^*)
\end{align*}
where we used the identities $[\partial_{x_i}\partial_{x_j}]=(x_i^*,x_j^*)=(f,g)=0;(x^*_i,g)=\partial_{x_i}g$ and $(x^*_j,f)=\partial_{x_j}f$.

The proof in the odd case is similar, although getting the signs right is a little painful. We have, taking into account the identities $\{\Pi x_i^*,\Pi x_j^*\}=\{f,g\}=0;\{\Pi x^*_i,g\}=\partial_{x_i}g$ and $\{\Pi x^*_j,f\}=\partial_{x_j}f$:
\begin{align*}
[\Dod(f\partial_{x_i}),\Dod(g\partial_{x_j})]=&(-1)^{|f|+|g|}\{f\Pi x_i^*,g\Pi x_j^*\}\\
=&(-1)^{|f|+|g|}\{f\Pi x_i^*,g\}\Pi x_j^*+(-1)^{(|f|+|x_i^*|)|g|+|f|+|g|}g\{f\Pi x_i^*,\Pi x_j^*\}\\
=&(-1)^{(|f|+|x^*_i|)|g|+|f|}\{g,f\Pi x_i^*\}\Pi x_j^*\\&+(-1)^{(|f|+|x^*_i|)|g|+(|f|+|x^*_i|+1)(|x^*_j|+1)+|f|+|g|}g\{\Pi x_j^*,f\Pi x_i^*\}\\
=&(-1)^{|g||f|+|g||x_i^*|+(|f|+|x^*_i|+1)|g|}f\{\Pi x^*_i,g\}\Pi x^*_j\\
&+(-1)^{(|f|+|x^*_i|)(|g|+|x_j^*|)+|x_i^*|+|x_j^*|+|g|+1}g\{\Pi x_j^*,f\Pi x_i^*\}\\
=&(-1)^{|g|}f\partial_{x_i}(g)\Pi x_j^*-(-1)^{(|f|+|x^*_i|)(|g|+|x_j^*|)+|x_i|+|x_j^*|+|g|}g\partial_{x_j}f\Pi x_i^*\\
=&(-1)^{|g|+|x_i^*|}\left((-1)^{|x_i^*|}f\partial_{x_i}(g)\Pi x_j^*-(-1)^{(|f|+|x^*_i|)(|g|+|x_j^*|)+|x_j^*|+|g|}g\partial_{x_j}f\Pi x_i^*\right)\\
=&(-1)^{|f|+|x_i^*|}\Dod[f\partial_{x_i}, g\partial_{x_j}]
\end{align*}
as desired.
\end{proof}

\begin{rem}
Associated to a (formal) function $h\in \hat{S}(V^*\oplus V)$ is a (formal) derivation $X_h$ of $\hat{S}(V^*\oplus V)$, the \emph{Hamiltonian vector field}, associated to $h$. The action of $X_h$ on $g\in \hat{S}(V^*\oplus V)$ is given by $X_h(g)= (-1)^{|h|}(h,g)$. The constants act by zero and so this gives an (of course, well-known) Lie algebra map $\phi_{\operatorname{ev}}\co\hat{S}_+(V^*\oplus V)\to\Der\hat{S}(V^*\oplus V)$. Similarly, there is
an odd Lie algebra map $\phi_{\operatorname{od}}\co\hat{S}_+(V^*\oplus \Pi V)\to\Der\hat{S}(V^*\oplus \Pi V)$. Composing the even and odd doubling maps with $\phi_{\text{ev}}$ and $\phi_{\operatorname{od}}$ we obtain the following two maps of Lie algebras:
\begin{gather*}
\Der(\hat{S}V^*)\to \Der\hat{S}(V^*\oplus V)\\
\Der(\hat{S}V^*)\to \Der\hat{S}(V^*\oplus \Pi V)
\end{gather*}
Abusing the notation, we shall refer to the latter maps as even and odd doubling maps respectively. Thus, in this interpretation the Lie algebra of vector fields on $V$ is realized as a Lie subalgebra in both $\Der\hat{S}(V^*\oplus V)$ and $\Der\hat{S}(V^*\oplus \Pi V)$.
\end{rem}

The construction of even and odd doubling maps gives rise to even and odd doubles of \emph{$L_\infty$~algebras}. Recall from Section~\ref{sec:linfty}, that an $L_\infty$~algebra structure on $V$ is an MC element in the Lie algebra $\Der_{\geq 2}(\hat{S}\Pi V^*)$. Using the correspondence between symplectic vector fields on $\Pi V$ and formal functions on $\Pi V$ we can view an even cyclic $L_\infty$~structure on $(V,d)$ endowed with an (even) inner product as an odd Hamiltonian $h\in\hat{S}\Pi V^*$ having no constant, linear or quadratic terms and satisfying the equation $dh+\frac{1}{2}(h,h)=0$. Similarly, an odd cyclic $L_\infty$~structure on $(V,d)$ endowed with an (odd) inner product is an even Hamiltonian $h\in\hat{S}\Pi V^*$ having no constant, linear or quadratic terms and satisfying the equation $dh+\frac{1}{2}\{h,h\}=0$.

Since MC elements are functorial with respect to Lie algebra maps, we can make the following definition.

\begin{defi}
Let $m\in \Der_{\geq 2}(\hat{S}\Pi V^*)$ be an $L_\infty$~algebra structure on the dg vector space $V$. The even double of $(V,m)$ is the $L_\infty$~structure $\Dev(m)\in \Der_{\ge 2}(\hat{S}\Pi(V^*\oplus V))$. The odd double of $(V,m)$ is the $L_\infty$~structure
$\Dod(m)\in \Der_{\geq 2}(\hat{S}\Pi(V^*\oplus \Pi V))$.
\end{defi}

\begin{rem}
The even and odd doubles of an $L_\infty$~algebra have naturally the structures of cyclic $L_\infty$~algebras (even and odd respectively).
\end{rem}

\begin{example}
Let us consider the case when $V$ is an \emph{ordinary} Lie algebra, possibly ungraded. The dual vector space $V^*$ has the structure of a $V$--module (the coadjoint module) and we can form
the \emph{square-zero extension} Lie algebra $V\oplus V^*$. The Lie bracket on $V\oplus V^*$ is specified by the formula:
\[
[v+ v^*,u+u^*]=[v,u]+v\cdot u^*-(-1)^{|v^*||u|}u\cdot v^*
\]
Here $v,u\in V, v^*,u^*\in V^*$, $[v,u]$ is the Lie bracket on $V$ and the expressions $v\cdot u^*$ and $u\cdot v^*$ denote the coadjoint actions of $v$ and $u$ on $u^*$ and $v^*$ respectively. The canonical inner product on $V\oplus V^*$ is clearly compatible with the Lie bracket. It is easy to see that the obtained Lie algebra structure on $V\oplus V^*$ is precisely $\Dev(V)$.

Similarly, $\Pi V^*$ is a $V$--module and we can similarly form the square-zero extension Lie algebra $V\oplus \Pi V^*$, supplied with an invariant \emph{odd} inner product. Again, it follows straightforwardly from definitions that the obtained Lie algebra is nothing but $\Dod(V)$. Thus, the notions of even and odd doubles are infinity versions of even and odd square-zero extensions of Lie algebras.
\end{example}

Let us now recall the notion of the \emph{divergence} of a formal vector field and of the \emph{Laplacian} of a (formal) function.

\begin{defi}\
\begin{enumerate}
\item
For a dg vector space $W$ the divergence of a vector field $\xi=\sum_{i}f_i\partial_{x_i}\in \Der\hat{S}W^*$ is defined as
\[
\nabla \xi=(-1)^{|f_i||x_i|}\sum_i\partial_{x_i}f_i\in \hat{S}W^*.
\]
\item
For $W=V^*\oplus\Pi V$ define the(odd) Laplacian $\Delta\co\hat{S}(W)\to \hat{S}(W)$ is the operator which acts on $g\in\hat{S}(W)$ according to the formula
\[
\Delta(g)=\frac{1}{2}\nabla(X_g).
\]
\end{enumerate}
\end{defi}

\begin{rem}
The definition of the divergence of a vector field and that of the Laplacian do not depend on the choice of a basis in $V$.
\end{rem}
Choose a basis $x_1,\dots, x_n$ in $V^*$ and the odd dual basis $\Pi x_1^*,\dots, \Pi x_n^*\in\Pi V$. Let $g\in \hat{S}(V^*\oplus \Pi V)$. Then direct inspection gives the following familiar formula:
\[
\Delta(g)=\sum_{i=1}^n\partial_{x_i}\partial_{\Pi x_i^*}g
\]
The following useful formula for the divergence of the commutator of two vector fields $\xi$ and $\eta$ could also be verified by direct inspection:
\begin{equation}\label{divergcom}
\nabla[\xi,\eta]=\xi\nabla\eta-(-1)^{|\xi||\eta|}\eta\nabla\xi.
\end{equation}

\begin{prop}\label{divlap}
Let $\xi\in\Der(\hat{S}V^*)$ be a formal vector field of $V$. Then $\Delta\left(\Dod(\xi)\right)$ belongs to $\hat{S}V^*\subset\hat{S}(V^*\oplus\Pi V)$ and the following identity in $\hat{S}V^*$ holds:
\[
\nabla(\xi)=\Delta\left(\Dod(\xi)\right)
\]
\end{prop}

\begin{proof}
Choose a basis $x_1,\dots, x_n$ in $V^*$ and the odd dual basis $\Pi x_1^*,\dots, \Pi x_n^*\in\Pi V$. It suffices to consider the case $\xi=f\partial_{x_i}$ for some $i=1,\dots, n$. Then $\Dod(\xi)=(-1)^{|f|}f\Pi x_i^*$. We have
\[
\nabla(\xi)=(-1)^{|f||x_i|}\partial_{x_i}f=\partial_{x_i}\partial_{\Pi x_i^*}\left((-1)^{|f|}f\Pi x_i^*)\right)=\Delta\left((-1)^{|f|}f\Pi x_i^*\right)
\]
as required.
\end{proof}

\begin{rem}
Proposition~\ref{divlap} can be reformulated as the identity $\nabla(X_{\Dod(\xi)})=2\nabla(\xi)$. Note that in the even case one always has $\nabla(X_{\Dev(\xi)})=0$.
\end{rem}

\section{Unimodular $L_\infty$~algebras}
For a dg vector space $(V,d)$ consider the following graded Lie algebra
\[
\g[V]:=\Der_{\geq 2}(\hat{S}V^*)\ltimes\Pi\hat{S}_{\geq 1}V^*,
\]
the semidirect product of $\Der_{\geq 2}(\hat{S}V^*)$ and $\Pi\hat{S}_{\geq 1}V^*$. The Lie bracket is defined by the formula
\[
[(\xi,\Pi f),(\eta,\Pi g)]=\left([\xi,\eta],\Pi\xi(g)+(-1)^{(|f|+1)|\eta|}\Pi\eta(f)\right).
\]
where $\xi,\eta\in\Der(\hat{S}V^*)$, $f,g\in\hat{S}V^*$. The differential $d$ in $V$ induces one in $\g[V]$ acting as by the Lie derivative along the linear vector field $d$. Furthermore, introduce an \emph{external} differential $d_e$ in $\g[V]$ by setting $d_e|_{\Pi\hat{S}V^*}=0$ and
\[
d_e(\xi)=\frac{1}{2}\Pi\nabla(\xi)\in\Pi\hat{S}_{\geq 1}V^*
\]
for $\xi\in\Der_{\geq 2}(\hat{S}V^*)$.
\begin{lem}\label{tracezero}
The operator $d+d_e$ in $\g[V]$ squares to zero.
\end{lem}
\begin{proof}
It is clear that each $d_e$ and $d$ squares to zero. It remains to prove that \[d\circ d_e+d_e\circ d=0.\] We only need to verify this identity on a formal vector field $\xi\in\Der_{\geq 2}(\hat{S}V^*)\hookrightarrow \g[V]$. Note that the linear map $d$ is traceless and thus, its corresponding linear vector field has zero divergence. Using this and formula (\ref{divergcom}) we compute:
\begin{align*}
(d\circ d_e+d_e\circ d)(\xi)&=d(\Pi\nabla\xi)+\Pi\nabla[d,\xi]\\
&=-\Pi d(\nabla\xi)+\Pi d(\nabla\xi)(-1)^{|\xi|}\Pi\xi(\nabla d)\\
&=0.
\end{align*}
\end{proof}
We can now define a unimodular $L_\infty$~algebra structure on the \emph{parity reversion} of $V$ as an MC element in $\g[V]$; equivalently such a structure on $V$ is an MC element in $\g[\Pi V]$. In more detail:
\begin{defi}
A unimodular $L_\infty$~algebra structure on $V$  is a pair $(m,f)$ consisting of an odd vector field $m\in\Der_{\geq 2}(\hat {S}\Pi V^*)$ and an even formal function $f\in\hat{S}\Pi V^*$, which satisfy the following equations:
\begin{align*}
&d(m)+\frac{1}{2}[m,m]=0;\\
&d(f)+\frac{1}{2}\nabla(m)+\xi(f)=0.
\end{align*}
Two unimodular $L_\infty$~algebra structures on $V$ are called \emph{equivalent} if they are equivalent as MC elements in the pronilpotent dgla $\g[\Pi V]$. In more detail:
$(m,f)$ is equivalent to $(m^\prime, f^\prime)$ if there exists an even vector field $\eta\in\Der_{\geq 2}(\hat {S}\Pi V^*)$ and and odd function $h\in \hat{S}_{\geq 1}\Pi V^*$ such that
\begin{equation}
m^{\prime}=e^{\eta}\cdot m :=e^{\eta}(m+d) e^{-\eta}-d;
\end{equation}
\begin{equation}\label{action}f^\prime=e^\eta\cdot f+(m+d)(h):=e^{\eta}(f)+(m+d)(h).
\end{equation}
\end{defi}

\begin{rem}\
\begin{enumerate}
\item
Sometimes it will be convenient for us to abuse the notation by referring to the first component $m$ of the pair $(m,f)$ as a unimodular $L_\infty$~algebra.
\item
A unimodular $L_\infty$~algebra can also be described as an algebra over the cobar construction of the wheeled closure of the operad $\Com$; cf.~\cite{Gran}. We will not use this point of view.
\item
The theory already exhibits all its properties when the internal differential $d$ on $V$ vanishes. In fact, choosing a split quasi-isomorphism $H(V)\hookrightarrow V$ we obtain an inclusion of dglas $\g[\Pi H(V)]\hookrightarrow \g[\Pi V]$ which is a filtered quasi-isomorphism. Thus, equivalence classes of unimodular algebras on $V$ and on $H(V)$ coincide. Moreover, this also implies that an $L_\infty$~structure on $V$ admits a unimodular lift if and only if the minimal model on $H(V)$ admits a unimodular lift (the theory of minimal models is recalled in Appendix~\ref{app:operads}).
\end{enumerate}
\end{rem}

\begin{defi}
The Lie subalgebra of $\Der(\hat{S}V^*)$ consisting of vector fields with zero divergence will be denoted by $\SDer(\hat{S}V^*)$. We will also use the notation
\[\SDer_{\geq 2}(\hat{S}V^*):=\SDer(\hat{S}V^*)\bigcap\Der_{\geq 2}(\hat{S}V^*).\]

A unimodular $L_\infty$~algebra structure on $V$ of the form $(m,0)$ is called \emph{strictly unimodular}. In that case $m$ necessarily belongs to $\SDer(\hat{S}\Pi V^*)$.
\end{defi}

It turns out that considering only strictly unimodular $L_\infty$~structures leads, with a small caveat, to no loss of generality. In preparation for this result we need the following elementary lemma showing that the divergence map is almost always surjective.

\begin{lem}\label{oddbehaviour}
Let $V$ be a vector space of dimension $l|n$. If $l>0$ then the map $\nabla\co\Der_{\geq 2}\hat{S}V^*\to\hat{S}_{\geq 1}V^*$ is surjective. If $l=0$ then the cokernel of $\nabla$ is one-dimensional. More precisely, if $y_1,\dots, y_n$ is a basis in the odd space $V^*$ then the element $y_1\cdot\ldots\cdot y_n\in \hat{S}V^*$ does not belong to the image of $\nabla$.
\end{lem}

\begin{proof}
Choose a basis $x_1,\dots,x_l,y_1,\dots,y_n$ in $V^*$ where the $x_i$s and $y_i$s are even and odd basis elements respectively. Suppose that $l>0$; to show that $\nabla$ is surjective it suffices to show that any monomial $f$ in $x_i$ and $y_i$ is in the image of $\nabla$. Suppose that $x_1$ enters into $f$ in the power $j\geq 0$. Then $f=\nabla \left(\frac{x_1}{j+1}f\partial_{x_1}\right)$ as required. Now suppose that $l=0$; consider a monomial $f$ in the odd variables $y_i$. If there exists $j=1,\dots, n$ for which $y_j$ does not enter into $f$ then $f=\pm\nabla (y_jf\partial_{y_j})$. It is also clear that the element $y_1\cdot \ldots \cdot y_n$ is not in the image of $\nabla$. The lemma is proved.
\end{proof}

\begin{theorem}\label{strict}
Let $(m,f)$ be a unimodular $L_\infty$~algebra structure supported on a dg vector space $V$ of dimension $l|n$.
\begin{enumerate}
\item
If $n>0$ or if $l$ is an odd number then $(m,f)$ is equivalent to a strictly unimodular $L_\infty$~structure.
\item
If $n=0$ and $l$ is an even number then $(m,f)$ is equivalent to a unimodular $L_\infty$~structure of the form $(m', c\cdot x_1 \cdots  x_l)$ where
$m'$ is a unimodular $L_\infty$~structure, $x_1,\dots, x_l$ is a basis in $\Pi V^*$ and $c$ is a number.
\end{enumerate}
\end{theorem}

\begin{proof}
Let $V$ be the underlying space of a unimodular $L_\infty$~algebra. Suppose first that $V$ is not purely even; then $\Pi V^*$ is not purely odd. In that case by Lemma~\ref{oddbehaviour} the map $\nabla\co\Der_{\geq 2}(\hat{S}\Pi V^*)\to \hat{S}_{\geq 1}\Pi V^*$ is onto. It follows that $\SDer_{\geq 2}(\hat{S}\Pi V^*)=\ker\nabla\co\Der_{\geq 2}(\hat{S}\Pi V^*)\to \hat{S}_{\geq 1}\Pi V^*$ is quasi-isomorphic to the dgla $\g[\Pi V]$. Note that $\g[\Pi V]$ has a filtration: $F_p\g[\Pi V]=\Der_{\geq p+1}(\hat{S}\Pi V^*)\ltimes \hat{S}_{p}V^*$ and the inclusion $i\co\SDer_{\geq 2}(\hat{S}\Pi V^*)\hookrightarrow \g[\Pi V]$ is compatible with this filtration and is a filtered quasi-isomorphism. It is well known (cf.~for example, \cite[Theorem 2.1]{Get}) that in such a situation the map $i$ induces a bijection between MC moduli sets of $\SDer_{\geq 2}(\hat{S}\Pi V^*)$ and $\g[\Pi V]$. The desired conclusion follows.

The case when $V$ is even is not much different; let $\dim V=l|0$. Then $\dim \Pi V^*=0|l$; choose an odd basis $x_1,\dots, x_l$ in this space. The dgla $\g[\Pi V]$ is quasi-isomorphic to the graded Lie algebra $\SDer_{\geq 2}(\hat{S}\Pi V^*)\oplus \Pi\langle x_1\cdot \ldots \cdot x_l\rangle$. Here $\langle x_1\cdot \ldots\cdot x_l\rangle$ denotes a one-dimensional space spanned by the central element $x_1\cdot \ldots\cdot x_l$. If $l$ is odd then any MC element in $\SDer_{\geq 2}(\hat{S}\Pi V^*)\oplus \Pi\langle x_1\cdot \ldots \cdot x_l\rangle$ has the form $(m^\prime,0)$ where $m'$ is an MC element in $\SDer_{\geq 2}(\hat{S}\Pi V^*)$ (and so is a unimodular $L_\infty$~structure). If $l$ is even then, in addition, the element $x_1\cdot \ldots\cdot x_l$ is MC and so, an arbitrary MC element is a pair $(m^\prime, c(x_1\cdot \ldots \cdot x_l))$ where $m'$ is an MC element in $\SDer_{\geq 2}(\hat{S}\Pi V^*)$ and $c$ is any scalar.
\end{proof}

\begin{rem}\
\begin{itemize}
\item
Another point of view on the above result is afforded by the notion of the Berezin volume form, \cite{Man}. Namely, let $m \in\Der_{\geq 2}(\hat{S}\Pi V^*)$ be an $L_\infty$~algebra structure on $V$ and $f\in \hat{S} \Pi V^*$. Then the condition that $(m,f)$ determines a unimodular $L_\infty$~algebra is equivalent to the statement that $m$ preserves the volume form $e^f \mu$ where $\mu$ is a standard constant volume form given by choosing linear coordinates on $\Pi V$. It is not hard to prove that if $V$ is not purely even or if the even dimension of $V$ is an odd number, then any two volume forms on $\Pi V$ are equivalent via a formal diffeomorphism and thus, one can take $f=0$.
\item
The case of a purely even space $V$ whose dimension is an even number is slightly different. Let us call such a dimension \emph{exceptional}. Theorem~\ref{strict} effectively shows that the moduli space of unimodular $L_\infty$~structures on a space of exceptional dimension is the direct product of the moduli space of strict unimodular structures on $V$ and the ground field $\ground$.
\item
The notion of a unimodular $L_\infty$~algebra should be compared to that of a \emph{symplectic} $L_\infty$~algebra, cf.~\cite{HL'} concerning this notion. In the case when a symplectic structure is the standard linear one (which can always be achieved by the Darboux theorem), symplectic $L_\infty$~algebras are usually called cyclic $L_\infty$~algebra.
\end{itemize}
\end{rem}
\begin{cor}\label{liftstrict}
An $L_\infty$~algebra $(V,m)$ admits a unimodular lift if and only if it is equivalent to a strictly unimodular $L_\infty$~algebra.
\end{cor}
\begin{proof}
If $(V,m)$ is equivalent to a strictly unimodular $L_\infty$~algebra $(V,m')$ then since the latter admits a unimodular lift, namely the trivial one, so does $(V,m)$ by Proposition \ref{prop:gaugeequivoffibres}. Conversely, if $(V,m)$ admits a unimodular lift, at worst it must be equivalent to a unimodular $L_\infty$~structure of the form $(m',c(x_1\cdot\ldots\cdot x_l))$ as in Theorem \ref{strict}. Note that the $L_\infty$~structures $m$ and $m'$ are equivalent. But since $x_1\cdot\ldots\cdot x_l$ is a central element in the dgla $\g[V]$, the pair $(m',0)$ is also an MC element in it and so $m'$ is in fact a strictly unimodular $L_\infty$~structure.
\end{proof}

\begin{rem}
IF $(V,m)$ is $L_\infty$~quasi-isomorphic to a strictly unimodular $L_\infty$~algebra then, by passing to minimal models, we see that $(V,m)$ admits a unimodular lift and is therefore in fact equivalent (i.e.~pointed $L_\infty$~isomorphic) to a strictly unimodular $L_\infty$~structure on $V$.
\end{rem}

\begin{rem}
There is another criterion for an $L_\infty$~algebra to be equivalent to a strictly unimodular one. Recall from \cite[Section 7]{Gran} that associated to an $L_\infty$~structure $m\in\Der(\hat{S}\Pi V^*)$ is a characteristic class $[\nabla(m)]\in \HCE(V)$ which vanishes if and only if $\xi$ admits a unimodular lift. Thus, the vanishing of $[\nabla(m)]$ is equivalent to the condition that $V$ is $L_\infty$~quasi-isomorphic to a strictly unimodular $L_\infty$~algebra.
\end{rem}

\begin{example}\label{nonunim}
Let $V$ be an ordinary graded Lie algebra. It is easy to see that $V$ is unimodular if and only if the adjoint action on $V$ by any element in $V$ has zero trace. In that case the notions of unimodularity and of strict unimodularity are the same. This property holds if $V$ is a semisimple or nilpotent Lie algebra. The smallest example of a non-unimodular Lie algebra is given by a Lie algebra spanned by two vectors $x,y$ in degree zero with $[x,y]=y$.
\end{example}

\section{Spaces of unimodular lifts}\label{spaceoflifts}
Given an $L_\infty$~algebra $(V,m)$ one can ask whether it can be lifted to a unimodular $L_\infty$~algebra $(m,f)$ and if so, in how many ways. We have seen (cf.~\cite{Gran}) that the existence of a lifting is equivalent to the vanishing of the characteristic class $[\nabla(m)]$ or, equivalently (Corollary~\ref{liftstrict}) $(V,m)$ is $L_\infty$~equivalent to a strictly unimodular $L_\infty$~algebra. The purpose of this subsection is to give a homological characterization of the space of unimodular lifts. The result is, in principle, a consequence of Theorem~\ref{thm:fibration} but we will give here a simple and direct treatment. We restrict to considering lifts of \emph{strictly unimodular} $L_\infty$~algebras. Since any $L_\infty$~algebra which admits a unimodular lift is $L_\infty$~equivalent to a strictly unimodular one according to Corollary~\ref{liftstrict}, this restriction results in no loss of generality by Proposition~\ref{prop:gaugeequivoffibres}.

\begin{defi}
Let $(V,m)$ be a strictly unimodular $L_\infty$~algebra and let $(m,f)$ and $(m,g)$ be unimodular lifts of $m$. Then $(m,f)$ and $(m,g)$ are called equivalent, if they are gauge equivalent as MC elements in the dgla $\g[\Pi V]$.
\end{defi}

Recall that an $L_\infty$~automorphism of an $L_\infty$~algebra $((V,d),m)$ is an automorphism $F\in\Aut(\hat{S}\Pi V^*)$ such that $F\circ (m+d)\circ F^{-1}=m+d$. Excluding automorphisms with a nontrivial reductive part we obtain a pronilpotent group of \emph{pointed} $L_\infty$~automorphisms which will be denoted by $\sAut_{L_\infty}(V)$. It is easy to see that any pointed automorphism of $(V,m)$ has the form $e^{\xi}$ where $\xi\in\Der_{\geq 2}(\hat{S}\Pi V^*)$ with $d(\xi)+[m,\xi]=0$.

\begin{defi}
The quotient of the group $\sAut_{L_\infty}(V)$ by the normal subgroup consisting of elements of the form $e^{d(\eta)+[m,\eta]}$ with $\eta\in\Der_{\geq 2}(\hat{S}\Pi V^*)$ will be called the group of \emph{homotopy $L_\infty$~automorphisms} of $V$ and will be denoted by $\HAut_{L_\infty}(V)$.
\end{defi}

\begin{rem}
Consider the complex $\CE_{\geq 2}(V,V)$; it is a clearly a pronilpotent dgla and the space of cocycles $\operatorname{ZCE}_{\geq 2}(V,V)$ forms a graded Lie subalgebra in it; moreover the cohomology $\HCE_{\geq 2}(V,V)$ also has the structure of a graded Lie algebra. Then $\sAut_{L_\infty}(V)$ is the pronilpotent group corresponding the the pronilpotent Lie algebra $\operatorname{ZCE}^0_{\geq 2}(V,V)$ whereas $\HAut_{L_\infty}(V)$ is the group corresponding to $\HCE^0_{\geq 2}(V,V)$.
\end{rem}

By definition the group $\Aut_{L_\infty}(V)$ acts on $\hat{S}\Pi V^*$, however we need to consider a twisted action, defined by formula (\ref{action}). This twisted action naturally lifts to an action of the group $\HAut_{L_\infty}(V)$ on $\HCE(V)$.

\begin{rem}
The twisted action defined above has the following natural interpretation; let $\mu$ be the constant Berezin volume form on $\Pi V$ corresponding to a linear structure on $\Pi V$. Then for $f\in \hat{S}\Pi V^*$ the action of a formal diffeomorphism $e^\xi$ on $f \mu$ is as follows: $e^{\xi}(f \mu)=(e^{\xi}\cdot f) \mu$.
\end{rem}

\begin{theorem}\
\begin{enumerate}
\item
Let $(V,m)$ be a strictly unimodular $L_\infty$~algebra. Then the correspondence $f\mapsto (m,f)$ establishes a 1--1 correspondence between cocycles $f\in \operatorname{ZCE}^0_{\geq 1}(V)$ and $L_\infty$~unimodular structures on $V$ lifting $m$.
\item
Two such structures $(m,f)$ and $(m,g)$ are equivalent if and only if there exists a pointed $L_\infty$~automorphism $F\in\Aut(\hat{S}\Pi V^*)$ such that the CE cocycles $f$ and $F \cdot g$ are cohomologous.
\item
The set of equivalence classes of unimodular lifts of $(V,m)$ is in 1--1 correspondence with the set of orbits of $\HCE^0_{\geq 1}(V)$ with respect to the twisted action of the group $\HAut_{L_\infty}(V)$.
\end{enumerate}
\end{theorem}

\begin{proof}
A pair $(m,f)$ is a unimodular $L_\infty$~structure if and only if $d(f)+\frac{1}{2}\nabla(m)+m(f)=0$ which in the strictly unimodular case simply means that $(d+m)(f)=0$. i.e.~that $f$ is a CE cocycle; this proves (1). Given two unimodular lifts $(m,f)$ and $(m,g)$ they are equivalent if and only if there is an even element $(\xi,h)\in\g[\Pi V]$ such that $e^{(\xi,h)}\cdot (m,f)=(m,g)$. We have:
\[
e^{(\xi,h)}\cdot (m,f)=\left(e^\xi \circ(m+d)\circ e^{-\xi}-d,e^\xi\cdot f+(m+d)(h)\right),
\]
which implies (2). Part (3) follows from the observation that the (twisted) action of $\sAut_{L_\infty}(V)$ on $\operatorname{ZCE}^0(V)$ factors through the quotient $\HAut_{L_\infty}(V)$.
\end{proof}

We have the following corollary, which can be expressed solely in terms of strictly unimodular $L_\infty$~algebras. To formulate it, we need the following definition.

\begin{defi}
A formal automorphism $f\in\Aut(\hat{S}\Pi V^*)$ is called \emph{unimodular} if it is of the form $f=e^{\xi}$ where $\xi\in\Der_{\geq 2}(\hat{S}\Pi V^*)$ and $\nabla(\xi)=0$.
\end{defi}

In other words, a unimodular automorphism is an automorphism whose linear part is the identity and which is volume-preserving. Note that any unimodular lift of a strictly unimodular $L_\infty$~algebra $(V,m)$ is $L_\infty$~equivalent to a strictly unimodular one. Conversely, any strictly unimodular $L_\infty$~structure $m^\prime$ on $V$ that is $L_\infty$~equivalent to $m$ is equivalent (as a unimodular $L_\infty$~structure) to a unimodular lift of $(V,m)$. Furthermore, two strictly unimodular $L_\infty$~structures on $V$ are equivalent if and only if they are equivalent through a strictly unimodular $L_\infty$~map. We have, therefore, the following result.

\begin{cor}
Let $(V,m)$ be a strictly unimodular $L_\infty$~algebra. The set of strict unimodular equivalence classes of strictly unimodular $L_\infty$~algebras on $V$ which are $L_\infty$~equivalent to $(V,m)$ is in bijective correspondence with $\HCE^0_{\geq 1}(V)/\HAut_{L_\infty}(V)$.\qed
\end{cor}

\begin{example}\
\begin{enumerate}
\item
Let $m$ be the zero $L_\infty$~structure on $V$; it is, of course, strictly unimodular. The CE complex of $(V,m)$ has trivial differential; the group of pointed $L_\infty$~automorphisms of $V$ is simply the group $G$ of automorphisms of $\hat{S}_{\geq 1} \Pi V^*$ whose linear part is the identity. The moduli space of unimodular lifts of $m$ is identified with the set of even formal functions $f\in \hat{S}_{\geq 1} \Pi V^*$ modulo the twisted action of $G$. The classification of the corresponding orbits is, of course, an intractable problem.
\item
By contrast, let the space $V$ be even and consider a semisimple Lie algebra on $V$. It is easy to see that it is unimodular and that the group $\HAut_{L_\infty}(V)$ is trivial. Thus, the moduli space of unimodular lifts coincides with even elements in $\HCE_{\geq 1}(V)$.
\end{enumerate}
\end{example}

\section{Quantum lifts of cyclic $L_\infty$~algebras}
In this section we consider the notion of a quantum $L_\infty$~algebra and the problem of quantum lifting of cyclic $L_\infty$~algebras.

Let $(V,d)$ be a dg vector space with an odd symplectic form. Choose a pair of complementary Langrangian spaces $U,W\subset V$, the given symplectic form determines a nondegenerate pairing between them and so we have an isomorphism
\[
V\cong U\oplus W\cong U^*\oplus U.
\]
It follows that there is defined an operator $\Delta$ on the space of formal functions on $V$; it is clear that it does not depend on the choice of complementary Lagrangian subspaces (e.g. one can take for $U$ and $W$ the subspaces of even and odd vectors in $V$ respectively). 
 Next, we introduce the weight grading on the cdga $\hat{S} V^*[[h]]$ by requiring that for a monomial $f\in \hat{S} V^*$ of degree $n$ the element $fh^g$ has weight $2g+n$. Let $\h[V]\subset \hat{S} V^*[[h]]$ be the subspace of weight grading $>2$. The cdga $\hat{S} V^*[[h]]$ has the standard structure of an odd Poisson algebra and $\h[V]$ inherits this structure; in particular it is a dgla with respect to the odd bracket $\{,\}$. It is clear that $\h[V]$ is pronilpotent, as opposed to $\hat{S} V^*[[h]]$. As usual, we let $d$ act on $\h[V]$ by the Lie derivative.

\begin{lem}
The operator $d+h\Delta$ endows $\h[V]$ with the structure of an odd dgla.
\end{lem}
\begin{proof}
The operators $d$ and $\Delta$ separately square to zero. It suffices to show that $d\Delta+\Delta d=0$; expressing $\Delta$ in terms of the divergence of the corresponding Hamiltonian vector field, we reduce the calculation to the one performed in the proof of Lemma \ref{tracezero}.
\end{proof}
A quantum $L_\infty$ structure on the \emph{parity reversion} of $V$ is an MC element in $\h[\Pi V]$; equivalently such a structure on $V$ is an MC element in $\h[V]$. In more detail:
\begin{defi}
Let $(V,d)$ be a dg vector space with an odd non-degenerate symmetric bilinear form. An even element $S(h)= S_0 + hS_1 + h^2S_2 + \dots \in \h[\Pi V]$ satisfying the MC equation
\begin{equation}\label{master}
(d+h\Delta )S(h)+\frac{1}{2}\{S(h),S(h)\}=0
\end{equation}
is called the structure of a \emph{quantum $L_\infty$~algebra} on $V$.
\end{defi}

\begin{rem}
The MC equation \eqref{master} is known under the name `quantum master equation' (QME). The corresponding algebraic structure, which we call `quantum $L_\infty$~algebra' appeared in the work of Zwiebach \cite{Zwi} on closed string field theory. It was further studied by Markl under the name `loop homotopy algebras', \cite{Mar}. The restriction on the weight, defining the dgla $\h[V]$ is a manifestation of the so-called \emph{stability condition} for modular operads, cf.~\cite{GK}. This allows one to interpret quantum $L_\infty$~algebras as algebras over the Feynman transform of the modular closure of the cyclic operad $\Com$ governing commutative algebras. We will not use this interpretation here.
\end{rem}

Let $S(h)=S_0+hS_1+h^2S_2+\dots$ be a solution of the QME \eqref{master} in $\h[\Pi V]$. The constant term in $h$ of the QME gives $dS_0+\frac{1}{2}\{S_0,S_0\}=0$. Therefore $X_{S_0}$ is an odd cyclic $L_\infty$~algebra structure on $V$; this is the genus zero part of the corresponding quantum $L_\infty$~algebra. Furthermore, the $h$--linear term in the QME gives \[dS_1+\Delta(S_0)+\{S_0,S_1\}=0.\] Since $\Delta(S_0)=\frac{1}{2}\nabla(X_{S_0})$, the pair $(X_{S_0},S_1)$ determines a unimodular $L_\infty$~structure on $V$.

Next, given an odd cyclic $L_\infty$~algebra structure $S_0$ on $V$ one can ask whether it is the genus zero part of a quantum $L_\infty$~algebra. Since
\begin{equation}\label{MCodd}dS_0+\frac{1}{2}\{S_0,S_0\}=0\end{equation}
the operator $?\mapsto d(?)+\{S_0,?\}$ has square zero on $\hat{S}\Pi V^* $. Moreover, the dg vector space $\hat{S}\Pi V^*$ supplied with this operator is precisely $\CE(V)$, the Chevalley--Eilenberg complex of $V$ regarded as an $L_\infty$~algebra.

Applying $\Delta$ to \eqref{MCodd} and using commutativity of $d$ and $\Delta$ we obtain
\[
d\Delta S_0+\{S_0,\Delta{S_0}\}=0.
\]
Thus, $\Delta(S_0)$ is a cocycle in $\CE(V)$; it is cohomologous to zero precisely when there exists $S_1\in \hat{S}\Pi V^*$ such that $dS_1+\Delta(S_0)+\{S_0,S_1\}=0$. All told, we obtain the following result.

\begin{prop}\label{obstr}
Let $S(h)$ be a quantum $L_\infty$~algebra structure on $V$. Then $X_{S_0}$ is an odd cyclic $L_\infty$~structure on $V$ and the pair $( X_{S_0},S_1 )$ is a unimodular $L_\infty$~structure on $V$. Moreover, an odd cyclic $L_\infty$~algebra structure on $V$ lifts to order one to a quantum $L_\infty$~algebra if and only if it is unimodular.\qed
\end{prop}

\begin{prop}
Let $m,m'$ be two odd cyclic $L_\infty$~structures on $V$ which are cyclic $L_\infty$~isomorphic. Then the sets of quantum lifts of $m$ and $m'$ are isomorphic.
\end{prop}

\begin{proof}
Saying that $m,m'$ are cyclic $L_\infty$~isomorphic is equivalent to saying that the $m,m'$ are gauge equivalent MC elements in the dgla of cyclic derivations $\Der_{\geq 2}^\cyc(\hat{S}V^*)$, cf.~Remark~\ref{rem:linftygaugeiso}. But now the result follows by Proposition~\ref{prop:gaugeequivoffibres}.
\end{proof}

\begin{rem}
Since a quantum $L_\infty$~structure $S=S_0+hS_1+\dots$ is an MC element in a certain dgla there is a corresponding notion of equivalence, or infinity isomorphism of two quantum $L_\infty$~structures defined on the same space $V$. Note that even though the $L_\infty$~algebra $(V,X_{S_0})$ is equivalent to a \emph{strict} unimodular $L_\infty$~structure, it does not imply that any quantum $L_\infty$~structure is equivalent to one having unimodular genus zero part since the notion of equivalence is different.
\end{rem}

In general, it is not so easy to decide whether an odd cyclic $L_\infty$~algebra $m=X_{S_0}$ can be lifted to a quantum $L_\infty$~algebra. The first obstruction to such a lift is provided, according to Proposition~\ref{obstr}, by the cohomology class $\Delta(S_0)$. If $S_0$ is harmonic, i.e.~$\Delta(S_0)=0$ on the nose (as opposed to up to a coboundary) then a constant lift $S=S_0$ is possible. We can give a complete answer in the case of the odd double of an $L_\infty$~algebra.
\begin{theorem}\label{quantunim}
Let $(V,m)$ be an $L_\infty$~algebra structure on $V$ and $(V\oplus\Pi V^*,X_{\Dod(\xi)})$ be its odd double. Then the following statements are equivalent.
\begin{enumerate}
\item The $L_\infty$~structure $(V,m)$ admits a unimodular lift.
\item The odd double $(V\oplus\Pi V^*,X_{\Dod(m)})$ of $V$ admits a unimodular lift.
\item The odd double $(V\oplus\Pi V^*,X_{\Dod(m)})$ of $V$ admits a quantum lift.
\end{enumerate}
\end{theorem}
\begin{proof}
Suppose that $(V,m)$ admits a unimodular lift. By Corollary \ref{liftstrict} it is equivalent to a strictly unimodular $L_\infty$~algebra $(V,m')$. Since $\Dod$ is an odd dgla map the double $(V\oplus\Pi V^*,X_{\Dod(m')})$ (which is strictly unimodular by Proposition~\ref{divlap}) is equivalent to $(V\oplus\Pi V^*,X_{\Dod(m)})$ and so the latter also admits a unimodular lift. This proves the implication $(1)\Rightarrow(2)$.

Suppose that the $L_\infty$~algebra $X_{\Dod(m)}$ admits a unimodular lift, so that
\begin{equation}\label{unimoddouble}
df+\frac{1}{2}\nabla(X_{\Dod(m)})+X_{\Dod(m)}(f)=0
\end{equation}
for some $f\in \hat{S}(\Pi V^*\oplus V)$. Note that $\Dod(m)\in \hat{S}\Pi V^*\otimes V \subset \hat{S}(\Pi V^*\oplus V)$. Therefore, writing $f=f_0+f_1+\dots$ where $f_i\in \hat{S}\Pi V^*\otimes V^{\otimes i}$ we see that
\[
df_i+X_{\Dod(m)}(f_i) = df_i + \{\Dod(m),f_i\}\in \hat{S}\Pi V^*\otimes V^{\otimes i}\subset\hat{S}(\Pi V^*\oplus V).
\]
Since $\nabla(X_{\Dod(m)})\in \hat{S}\Pi V^*\subset \hat{S}(\Pi V^*\oplus V)$ We conclude from (\ref{unimoddouble}) that
\[
df_0+\frac{1}{2}\nabla(X_{\Dod(m)})+X_{\Dod(m)}(f_0)=0.
\]
It follows from Proposition~\ref{divlap} that the pair $(m, f_0/2)$ is a unimodular $L_\infty$~structure lifting $m$ and so $(2)\Rightarrow(1)$. Furthermore $\Delta f_0 = 0$ and therefore $\Dod(m)+ hf_0$ is a quantum $L_\infty$~structure lifting $\Dod(m)$, proving $(2)\Rightarrow (3)$.

To complete the proof it remains to show that $(3)\Rightarrow (2)$, but this is immediate from Proposition~\ref{obstr}.
\end{proof}
\begin{rem}
In fact, the set of lifts of a given $L_\infty$~structure $m$ on $V$ to a unimodular $L_\infty$~structure on $V$ \emph{injects} into the set of quantum lifts of the odd $L_\infty$~structure $\Dod(m)$. Of course, the set of quantum lifts is potentially much bigger.
\end{rem}

\section{Unimodularity of tensor products}\label{sec:tensors}
It is known that a tensor product of an $L_\infty$~algebra and a cdga is itself an $L_\infty$~algebra. In this section we investigate the question of when the resulting tensor product admits a unimodular lift (which we know is the same as being equivalent to a strictly unimodular $L_\infty$~algebra). We start by recalling a general construction from \cite{HL}.

Let $V$ be a dg vector space and $A$ be a cdga. We construct a map of dglas
\begin{equation}\label{tensor}
\Psi_A\co\Der_{\geq 2}\hat{S}V^*\to\Der_{\geq 2}\hat{S}(A\otimes V)^*
\end{equation}
as follows. Any derivation $\xi\in \Der\hat{S}V^*$ is determined by its value on $V^*$ which could be any element in $\hat{S}V^*$.
If this element belongs to $\hat{S}^nV^*$ we say that $\xi$ is homogeneous of degree $n$. Without this restriction, therefore, we can write $\xi=\xi_0+\xi_1+\dots$ where $\xi_n$ is the $n$--th homogeneous component of $\xi$. Any homogeneous derivation $\xi$ of degree $n$ can be identified with a map $V^*\to \hat{S}^nV^*$. We have a symmetrization isomorphism $i_n\co\hat{S}^nV^*\to({S}^nV)^*$
(cf.~Section~\ref{sec:linfty}) and dualizing, we obtain a map ${S}^nV\to V$, i.e.~a multi-linear symmetric map $\tilde{\xi}\co S^nV\to V$. We thus have for $v_1,\dots,v_n\in V$:
\[
\tilde{\xi}(v_1,\dots, v_n)=\xi^*\left(\sum_{\sigma\in S_n}\sigma(v_1\otimes\dots\otimes v_n)\right)
\]
\begin{example}
Take $V$ to be the one-dimensional even space $\ground$ so that $\hat{S}V^*\cong\ground[[t]]$ and consider $\xi=t^n\partial_t$ for $n\geq 0$. The map $\tilde{\xi}\co\ground\to \ground$ is multiplication by $n!$.
\end{example}
There is clearly a one-to-one correspondence between such multi-linear maps and homogeneous derivations $\xi\in \Der\hat{S}W^*$ of degree $n$. We define the map $\Psi_A$ on derivations of a fixed homogeneous degree $n>0$ and then extend by linearity. Denote by $X^n\co A^{\otimes n}\to A$ the $n$--th iterated multiplication map, so $X^n(a_1,\dots, a_n)=a_1\cdot\ldots\cdot a_n$.
\begin{defi}
For $\xi\in\Der_{n}\hat{S}V^*$ the derivation $\Psi_A(\xi)$ is the derivation of $\hat{S}(A\otimes V)^*$ which corresponds to the multilinear map
\[
X^n\otimes \tilde{\xi}\co A^{\otimes n}\otimes V^{\otimes n}\to A\otimes V.
\]
\end{defi}
\begin{prop}\label{Liealgmap}
The map (\ref{tensor}) is a map of dglas.
\end{prop}
\begin{proof}
A detailed direct proof is given in \cite[Lemma 6.4]{HL}. Here we sketch a more conceptual, non-computational proof. Denote by $\h(V)$ the graded Lie algebra $\Der_{\geq 2}\hat{S} V^*$; similarly $\h(A\otimes V)$ will stand for the graded Lie algebra $\Der_{\geq 2}\hat{S}(A^*\otimes V^*)$. We wish to prove that the map $\Psi_A\co\h(V)\to\h(A\otimes V)$ respects the dgla structure.

Note that an MC element in $\h(V)$ is an $L_\infty$~algebra structure on $\Pi V$ whereas an MC element in $\h(A\otimes V)$ is an $L_\infty$~structure on $A\otimes \Pi V$. Let $m^{\Pi V}$ be an $L_\infty$~structure on $\Pi V$; recall that we write $\tilde{m}_n\co V^{\otimes n}\to V$ for the corresponding multi-linear maps. Associated to $m^{\Pi V}V$ is an $L_\infty$~structure $m^{A\otimes \Pi V}$ on $A\otimes \Pi V$ given by a tensor product with the cdga $A$ (see Section~\ref{sec:linfty}).
This gives a map on MC elements $\MC(\h(V))\to\MC(\h(A\otimes V))$. This map can be extended functorially over any dg base; i.e.~for any cdga $B$ there is a natural map
$\MC\left(B\otimes\h(V)\right)\to\MC\left(B\otimes\h(A\otimes V)\right)$. But the existence of such a
map is tantamount to having an $L_\infty$~map $\h(V)\to\h(A\otimes V)$,
see \cite[Corollary 2.5]{ChL} for this type of argument. It is then easy to check that this map is precisely $\Psi_A$ and, in particular, that it is strict.
\end{proof}
A slight modification of the proof also yields a map of dglas $\Der\hat{S}V^*\to\Der\hat{S}(A\otimes V)^*$; we will not use this extended map. An interesting special case is obtained by setting $V=\ground$, the ground field. In that case we obtain a Lie algebra map $\Der\ground[[t]]\mapsto\Der \hat{S}A^*$ (which is easily seen to be injective). This can be phrased as saying that the Lie algebra of formal vector fields on the line has a (nonlinear) action on the underlying space on any commutative and associative algebra $A$. In fact, it is easy to see that in this case $A$ need not even be commutative; formal vector fields could also be replaced by polynomial ones. This observation was also made in \cite[Theorem 1.7.8]{KM}.

The following result expresses the divergence of a formal vector field in terms of traces; it was proved in \cite[Proposition 6.12]{HL}; the proof is just a straightforward unwrapping of definitions.
\begin{lem}\label{trace}
Let $V$ be a finite-dimensional vector space and $\xi\in\Der_n(\hat{S}V^*)$ and $\tilde{\xi}\co V^{\otimes n}\to V$ be the corresponding multi-linear map. Then the multi-linear map $i_{n-1}(\nabla\xi)\co S^{n-1}V\to\ground$ has the form:
\[v_1\otimes\dots\otimes v_{n-1}\mapsto
\Tr\left(x\mapsto\tilde{\xi}(v_1,\dots,v_{n-1},x)\right)
\]\qed
\end{lem}

The dgla map $\Psi_A$ can be extended to a dgla map $\g[V]\to\g[A\otimes V]$ where $V$ is a dg vector space and $A$ is a cdga. Consider the map
\[\hat{S}_{\geq 1}V^* \cong\Hom(S_{\geq 1}V,\ground)\to\hat{S}_{\geq 1}(A\otimes V)^*\]
which associates to a symmetric multilinear map $f\co V^{\otimes n}\to \ground$ the symmetric multilinear map $\tilde{f} (A\otimes V)^{\otimes n}\to \ground$ such that
\[
\tilde{f}(a_1\otimes v_1,\dots,a_n\otimes v_n)=\pm\Tr(a_1\dots a_n)f(v_1,\dots,v_n)
\]
for $a_i\in A, v_i\in V,i=1,\dots,n$; the sign $\pm$ is determined by the Koszul sign rule.

The map $f\mapsto \tilde{f}$ induces a map $\Pi\hat{S}_{\geq 1}V^*\to\Pi\hat{S}_{\geq 1}(A\otimes V)^*$ which we will denote by $\Psi^\prime_A$. Then we have the following result.
\begin{prop}\label{psi_Aext}
The map
\[(\Psi_A,\Psi_A^\prime)\co\Der_{\geq 2}\hat{S}V^*\ltimes\Pi\hat{S}_{\geq 1}V^*\to \Der_{\geq 2}\hat{S}(A\otimes V)^*\ltimes\Pi\hat{S}_{\geq 1}(A\otimes V)^*\]
is a map of dglas $\g[V]\to\g[A\otimes V]$.
\end{prop}
\begin{proof}
Let $\xi\in \Der_n(\hat{S}V^*)$ and let $\tilde{\xi}\co V^{\otimes n}\rightarrow V$,be the corresponding multi-linear map. Given a symmetric map $f\co V^{\otimes m}\to \ground$ then the symmetric map $\{\xi, f\}\co V^{\otimes n+m-1}\to \ground$ is given by the formula
\[
\{ \xi, f \}(v_1,\dots, v_{n+m-1}) = \sum_{i=1}^{m} \pm f(v_1,\dots, v_{i-1},\tilde{\xi}(v_i,\dots, v_{i+n}),\dots, v_{n+m-1})
\]
where $\pm$ is determined by the Koszul sign rule. By straightforward inspection using this formula and the formula for the map $(\Psi_A, \Psi_A')$, the Lie bracket is preserved. Similarly, commutation with the differential follows from the formula for the divergence in Lemma~\ref{trace}.
\end{proof}
We can now give a criterion for a tensor product of an $L_\infty$~algebra and a cdga to be a (strictly) unimodular $L_\infty$~algebra. First, a relevant definition.
\begin{defi}\label{def:commutativeunimodular}
A finite dimensional cdga $A$ is \emph{unimodular} if the operator of multiplication by any element $a\in A$ has zero trace.
\end{defi}
Note for future use that the map $\Psi^\prime_A$ is zero if the cdga $A$ is unimodular. Next, there is a simple criterion for a graded commutative algebra to be unimodular.
\begin{prop}
Let $A$ be a finite dimensional graded commutative algebra and $e_i,i=1,\dots, n$ be a complete set of orthogonal idempotents in $A$ so that $A\cong \oplus_{i=1}^ne_iA$. Then $A$ is unimodular if and only if for each $i=1,\dots, n$ the even and odd dimensions of $e_iA$ coincide, i.e.~there exist $l_i\in\mathbb Z$ such that $\dim(e_iA)=l_i|l_i$.
\end{prop}
\begin{proof}
Without loss of generality we assume that the ground field $\ground$ is algebraically closed. Clearly $A$ is unimodular if and only if every $e_iA$ is. Next, $e_iA$ is a local algebra with residue field $\ground$ and it is unimodular if and only if the trace of the multiplication by 1 in it is zero. If $\dim e_iA=l_i|n_i$ then this trace is equal to $l_i-n_i$ so it is zero if and only if $l_i=n_i$ as required.
\end{proof}
\begin{rem}\label{rem:cdgastrict}
The reader may feel that Definition~\ref{def:commutativeunimodular} is really a definition of \emph{strict} unimodularity. Indeed, there exists a notion of a unimodular $C_\infty$~algebra and regarding a cdga $A$ as a $C_\infty$~algebra, it can be shown that $A$ lifts to a unimodular $C_\infty$~algebra if and only if $H(A)$ is unimodular in the sense of Definition~\ref{def:commutativeunimodular}, although $A$ itself will not be unless it admits the trivial/zero lift. However, it turns out that there is no difference between unimodularity and strict unimodularity for cdgas. Indeed, assuming without loss of generality that $\ground$ is algebraically closed, we saw that the unimodularity of a cdga $A$ is equivalent to the condition that for any idempotent $e\in A$ we have $\Tr(x\mapsto ex)=0$. Note that $e$ is necessarily a cycle in $A$ and since the map $A\to \ground$ given by $a\mapsto \Tr (x\mapsto ax)$ is a chain map, the trace of the multiplication by $e$ is zero if and only if the trace of the induced map on $H(A)$ is also zero. Therefore $A$ is unimodular in the sense of Definition~\ref{def:commutativeunimodular} if and only if $H(A)$ is.
\end{rem}

\begin{example}\label{frobuni}\
Let $A=H^*(M)$ be the cohomology ring of an orientable manifold $M$. In that case $A$ is local and so it is unimodular if and only if the Euler characteristic of $M$ vanishes: $\chi(M)=0$. This is the case, e.g.~if $\dim M$ is odd. In fact, it is easy to see that any graded unital \emph{odd} Frobenius algebra (i.e.~having an invariant odd inner product) is necessarily unimodular.
\end{example}
\begin{theorem}\label{tenpr}
Let $A$ be a cdga and $V$ be an $L_\infty$~algebra. Then:
\begin{enumerate}
\item The $L_\infty$~algebra $A\otimes V$ is strictly unimodular if and only if either $V$ is strictly unimodular or $A$ is unimodular.
\item The $L_\infty$~algebra $A\otimes V$ admits a unimodular lift if and only if either $V$ admits a unimodular lift or $A$ is unimodular.
\end{enumerate}
\end{theorem}
\begin{proof}
Let $m\in \Der_{\geq 2}\hat{S}\Pi V^*$ be the given $L_\infty$~structure on $V$. Then by Proposition
\ref{psi_Aext}
we have \[\nabla\Psi_A(m)=\Psi^\prime_A(\nabla m).\] It follows that if $V$ is strictly unimodular, i.e.~if $\nabla m=0$ then $\nabla\Psi_A(m)=0$, i.e.~the $L_\infty$~algebra $A\otimes V$ is strictly unimodular. Taking into account that
$\Psi_A^\prime\co\hat{S}_{\geq 1}\Pi V^*\to\hat{S}_{\geq 1}\Pi (A\otimes V)^*$ is a chain map we conclude, similarly, that if $V$ admits a unimodular lift, i.e.~$\nabla m$ is a coboundary in $\hat{S}\Pi V^*$ then $\nabla\Psi_A(m)$ is a coboundary in $\hat{S}\Pi (A\otimes V)^*$, i.e.~$A\otimes V$ admits a unimodular lift.

Further, if $A$ is unimodular then the map $\Psi_A^\prime$ is zero and so again, the $L_\infty$~algebra $A\otimes V$ is unimodular. This proves the `if' statement of both (1) and (2).

Conversely, assume $A\otimes V$ admits a unimodular lift. For any idempotent $e\in A$, define the map $\phi_e\co \hat{S}\Pi (A\otimes V)^* \rightarrow \hat{S} \Pi V^*$ by:
\[
\phi_e(f)(v_1,\dots, v_n) = f(e\otimes v_1, \dots, e\otimes v_n).
\]
Then $\phi_e$ can be viewed as a strict $L_\infty$~map $V\to A\otimes V$; $v\mapsto e\otimes v$ and thus, it is a chain map with respect to the $\CE$ differentials. Furthermore, the following formula can be easily checked:
\[
\phi_e(\nabla \Psi_A) = \Tr(a\mapsto ea)(\nabla m).
\]
 But since $A\otimes V$ admits a unimodular lift then $\nabla \Psi_A\in \hat{S}\Pi (A\otimes V)^*$ is cohomologous to zero and hence so is $\Tr(a\mapsto ea)(\nabla m)\in \hat{S} \Pi V^*$. Therefore, either $\Tr(a\mapsto ea)=0$ for any idempotent $e\in A$ or $\nabla m$ is cohomologous to zero. In the first case $A$ is unimodular and in the second case $V$ has a unimodular lift.

If $A\otimes V$ is strictly unimodular then a similar argument shows that either $A$ is unimodular or $V$ is strictly unimodular.
\end{proof}
\section{Application to Chern--Simons theory}
Let $M$ be a smooth closed oriented manifold of dimension $n$. Its de~Rham algebra $\Omega^*(M)$ is a cdga with the Poincar\'e duality pairing: given two forms $\omega_1$ and $\omega_2$ their inner product is
\[
[ \omega_1,\omega_2]=\int_M\omega_1\cdot \omega_2.
\]
This pairing is invariant in the sense that the identity $[ \omega_1\omega_2,\omega_3]=[\omega_1,\omega_2\omega_3]$ holds for any three forms $\omega_1,\omega_2$ and $\omega_3$. Note that the pairing $[,]$ is even or odd depending on whether $n$ is even or odd. Of course $\Omega(M)$, being infinite-dimensional, is not a Frobenius algebra, however its cohomology $H^*(M)$ is.

Let $V$ be a \emph{finite-dimensional} cyclic $L_\infty$~algebra. We assume that the inner product $\langle, \rangle$ on $V$ has parity which is \emph{opposite} to that of $[,]$. Thus, if $\dim M$ is odd, $V$ is an even cyclic $L_\infty$~algebra and if $\dim M$ is even, then $V$ is an odd cyclic $L_\infty$~algebra. Consider the tensor product
$\Omega(M)\otimes V$. This is itself an $L_\infty$~algebra (as a tensor product of a cdga and an $L_\infty$~algebra) and it has an odd pairing that is the tensor product of pairings on $\Omega(M)$ and $V$, making it into an odd cyclic $L_\infty$~algebra. We will want to consider when this $L_\infty$~algebra admits a unimodular or quantum lift. However, this only makes sense for finite dimensional $L_\infty$~algebras. Therefore we replace this $L_\infty$~algebra with its \emph{minimal model}, which is a finite dimensional cyclic $L_\infty$~algebra which is homotopy equivalent to the original cyclic $L_\infty$~algebra, so still encodes its homotopy type.

We will also need to briefly consider cyclic $C_\infty$~algebras and their minimal models in this section and we refer to \cite{HL} for the relevant definitions. The theory of minimal models of cyclic $L_\infty$~algebras and cyclic $C_\infty$~algebras which we make use of is treated in detail in Appendix~\ref{app:operads}, in the generality of algebras over operads and cyclic operads. The general notion of homotopy equivalence for algebras over operads and cyclic operads relevant to this section is also introduced in Appendix~\ref{app:operads}.

We will make use of the following fact about $C_\infty$~algebras, which is a rewording of the main theorem of \cite{HL}. It says, essentially, that a minimal cyclic $C_\infty$~algebra is determined up to homotopy equivalence by the equivalence class of the underlying (non-cyclic) $C_\infty$~algebra and the isomorphism class of the underlying Frobenius algebra.

\begin{theorem}[{\cite[Theorem 13.5 (1)]{HL}}]\label{thm:hlthm}
Let $A$ be a unital commutative Frobenius algebra with zero differential. Then two minimal cyclic $C_\infty$~structures on $A$ which lift this Frobenius algebra structure are homotopy equivalent as cyclic $C_\infty$~algebras if and only if they are homotopy equivalent as (non-cyclic) $C_\infty$~algebras.
\end{theorem}

Let us denote by $\g(M,V)$ a cyclic $L_\infty$~minimal model of $\Omega(M)\otimes V$. It has $H(M)\otimes H(V)$ as its underlying graded vector space; it is thus finite-dimensional and the scalar product is non-degenerate.

\begin{prop}\label{cyclicmin}
Let $M$ and $N$ be two closed oriented manifolds which are homotopy equivalent through an orientation-preserving homotopy equivalence. Then the cyclic $L_\infty$~algebras $\g(M,V)$ and $\g(N,V)$ are cyclic $L_\infty$~isomorphic.
\end{prop}
\begin{proof}
The given homotopy equivalence $M\to N$ induces a quasi-isomorphism of cdgas $f\co\Omega(N)\to\Omega(M)$ which preserves the pairings: $[\omega_1,\omega_2]=[f(\omega_1),f(\omega_2)]$ for $\omega_1,\omega_2\in\Omega(N)$. It follows that the $C_\infty$~minimal models of $\Omega(M)$ and $\Omega(N)$ are homotopy equivalent (in the sense of Definition~\ref{def:homotopyequiv}) as (non-cyclic) $C_\infty$~algebras. Furthermore, $H(M)$ and $H(N)$ are isomorphic as Frobenius algebras. Therefore, by Theorem~\ref{thm:hlthm} the cyclic $C_\infty$~minimal models of $\Omega(M)$ and $\Omega(N)$ are homotopy equivalent as \emph{cyclic} $C_\infty$~algebras. Tensoring these cyclic $C_\infty$~minimal models with $V$ we obtain, by Theorem~\ref{thm:homotopyequiv}, two homotopy equivalent cyclic $L_\infty$~algebras. These cyclic $L_\infty$~algebras are minimal models for $\Omega(M)\otimes V$ and $\Omega(N)\otimes V$, so are homotopy equivalent to $\g(M,V)$ and $\g(N,V)$ by the uniqueness of minimal models up to homotopy, cf.~Theorem \ref{thm:minmod}, (2). Therefore $\g(M,V)$ and $\g(N,V)$ are homotopy equivalent and hence cyclic $L_\infty$~isomorphic, cf.~Remark~\ref{rem:inftyhomotopyequiv}.
\end{proof}
Thus, the cyclic isomorphism class of $\g(M,V)$ is a homotopy invariant of the manifold $M$. In fact, it is clear that one only needs $M$ to be a Poincar\'e duality space in order to define $\g(M,V)$. Using sophisticated renormalization techniques, Costello \cite{Cos} proved that the $L_\infty$~algebra $\g(M,V)$ admits a quantum lift $\g^q(M,V)$, which is determined up to homotopy by the smooth structure on $M$. The homotopy type of the resulting quantum $L_\infty$~algebra is, therefore, a smooth invariant of $M$. It is not known whether $\g^q(M,V)$ is a homotopy invariant of $M$ but its genus zero part, which is just the $L_\infty$~algebra $\g(M,V)$, certainly \emph{is} a homotopy invariant. We will now analyze the situation using the obstruction-theoretic approach.
To this end let us recall the following theorem due to Lambrechts--Stanley \cite{LS}, which we restate in our current language.
\begin{theorem}\label{LS}
Let $A$ be a unital $\mathbb{Z}$--graded cdga with a possibly degenerate invariant graded pairing and which is \emph{simply-connected}, meaning $H^0(A)=\ground$ and $H^1(A)=0$. Furthermore assume that the pairing on $H(A)$ is non-degenerate. Then there exists a unital $\mathbb{Z}$--graded cyclic cdga $A'$ with a \emph{non-degenerate} invariant pairing which is \emph{weakly equivalent} to $A$, meaning it can be connected to $A$ by a zig-zag of quasi-isomorphisms of cdgas.
\end{theorem}
\begin{rem}
A non-degenerate graded pairing on $A'$, is equivalent to a non-degenerate graded trace $A'\to \ground$. Because $A'$ is connected, any such trace corresponds uniquely to a fundamental cycle in $A'$, i.e.~a non-zero element of top degree, moreover it can be chosen arbitrarily. With a suitable choice of a fundamental cycle in $A'$ we can always achieve that the cohomology of $A$ and $A'$ are isomorphic as Frobenius algebra and we will assume this in what follows.
\end{rem}
It is not claimed that a Lambrechts--Stanley model is unique. However we have the following useful strengthening of the Lambrechts--Stanley theorem showing that it is unique \emph{up to a strong cyclic homotopy equivalence}.
\begin{prop}\label{LS2}
Let $A$ be a unital $\mathbb{Z}$--graded cyclic cdga which is simply connected and let $A'$ be a Lambrechts--Stanley model. Then the cyclic $C_\infty$~minimal models of $A$ and $A'$ are homotopy equivalent.
\end{prop}

\begin{proof}
Since $A$ and $A'$ are weakly equivalent as cdgas, it follows that the cyclic $C_\infty$~minimal models of $A$ and $A'$ are homotopy equivalent as (non-cyclic) $C_\infty$~algebras, moreover this homotopy equivalence induces an isomorphism of graded Frobenius algebra structures on $H(A)$ and $H(A')$. Then by Theorem~\ref{thm:hlthm} it follows that the minimal models are in fact homotopy equivalent as \emph{cyclic} $C_\infty$~algebras.
\end{proof}

Now, let $M$ be a simply-connected closed oriented manifold and consider the finite-dimensional cyclic cdga $\tilde{\Omega}(M)$ corresponding to $\Omega(M)$ by Theorem~\ref{LS}; according to our conventions we will view $\tilde{\Omega}(M)$ as a $\mathbb Z/2$--graded cdga. Let $V$ be a cyclic $L_\infty$~algebra of parity opposite to that of $\dim M$. By the homotopy invariance of tensoring with $V$ (see Theorem~\ref{thm:homotopyequiv}) we see that $\g(M,V)$ is cyclic $L_\infty$~isomorphic as a cyclic $L_\infty$~algebra to the cyclic minimal model of $\tilde{\Omega}(M)\otimes V$. Therefore $\g(M,V)$ lifts to a unimodular (or quantum) $L_\infty$~algebra if and only if $\tilde{\Omega}(M)\otimes V$ does.

Note that the cdga $\tilde{\Omega}(M)$ is local and so its only non-zero idempotent is the identity element $1$. It is unimodular if and only if the multiplication by $1$ has zero trace which in turn holds if and only if $\chi(M)=0$.
These arguments, together with Theorem~\ref{tenpr} yield the following result.
\begin{theorem}\label{CS1}
Let $M$ be a simply-connected closed oriented manifold and $V$ be a cyclic $L_\infty$~algebra whose parity is opposite to $\dim M$. Then $\g(M,V)$ has a unimodular lift if and only if either $V$ has a unimodular lift or $\chi(M)=0$.\qed
\end{theorem}

We can also analyze the existence of a \emph{quantum lift} of $\g(M,V)$.
\begin{cor}\label{CS2}
Let $M$ and $V$ be as in Theorem~\ref{CS1}. Then if either $\chi(M)=0$ or $V$ is strictly unimodular then $\g(M,V)$ admits a quantum lift. Conversely, if $\g(M,V)$ admits a quantum lift then either $\chi(M)=0$ or $V$ has a unimodular lift.
\end{cor}
\begin{proof}
If $\chi(M)=0$ then $\tilde{\Omega}(M)$ is a unimodular cdga and thus, by Theorem~\ref{tenpr} (1) the $L_\infty$~algebra $\tilde{\Omega}(M)\otimes V$ is strictly unimodular and thus, admits a trivial quantum lift. This implies that $\g(M,V)$ likewise admits a quantum lift. Similarly if $V$ is strictly unimodular then $\tilde{\Omega}(M)\otimes V$ is strictly unimodular which again, implies that $\g(M,V)$ admits a quantum lift.

Finally, assume that $\g(M,V)$ admits a quantum lift. Then it must also admit a unimodular lift and by Theorem~\ref{CS1} either $\chi(M)=0$ or $V$ has a unimodular lift.
\end{proof}
\begin{example}
Let $\g$ be an $L_\infty$~algebra which does not have a unimodular lift; e.g.~one can take an ordinary Lie algebra that is not unimodular, cf.~Example \ref{nonunim}. Then the odd double $\Dod(\g)$ is a cyclic $L_\infty$~algebra which does \emph{not} admit a unimodular lift by Theorem \ref{quantunim}. It follows that for a simply-connected manifold $M$ with $\chi(M)\neq 0$ the $L_\infty$~algebra $\Omega(M)\otimes \Dod(\g)$ does not have a unimodular lift. Since quantum liftability is a stronger condition, it does not admit a quantum lift either.
\end{example}
\begin{rem}\
\begin{enumerate}
\item If $M$ is a \emph{formal} manifold, meaning that the de~Rham algebra $\Omega(M)$ is quasi-isomorphic as a cdga to its cohomology $H(M)$ (for example, if $M$ is a two dimensional surface, or a K\"ahler manifold), then we may drop the assumption that $M$ is simply connected from Theorem~\ref{CS1} and Corollary~\ref{CS2}, since we no longer need the Lambrechts--Stanley result in this case. Indeed, we just use the Frobenius algebra $H(M)$ for the algebra $\tilde{\Omega}(M)$ in the arguments above.
\item
If $\dim M$ is odd (as in the traditional setup of Chern--Simons theory) then $\chi(M)=0$ and it follows that $\g(M,V)$ admits a quantum lift (and thus, also a unimodular lift).
\item
We proved, therefore, that (in favorable circumstances) the cyclic $L_\infty$~algebra $\g(M,V)$ has a quantum lift. However our methods do not extend to constructing a \emph{canonical} lift. The existence of a canonical lift was claimed by Costello in \cite{Cos}. The necessity of $V$ being unimodular in the case $\dim M$ is even and $\chi(M)\neq 0$ was omitted in op.~cit.
\end{enumerate}
\end{rem}

\appendix

\section{Maurer--Cartan elements and lifts}\label{app:mc}
The purpose of this appendix is to review the theory of Maurer--Cartan (MC) elements in dglas and their moduli. The material here is more or less standard, perhaps excepting the theory of lifts of MC elements. A more detailed and general overview is contained in \cite{braun}.

\begin{defi}
Let $\g$ be a differential graded Lie algebra. An \emph{MC element in $\g$} is a degree one element $\xi\in\g$ satisfying the Maurer--Cartan equation
\begin{equation}\label{MCdef}
d\xi + \frac{1}{2}[\xi, \xi]=0.
\end{equation}
We denote the set of MC elements in $\g$ by $\MC(\g)$.
\end{defi}

Since a map of dglas $\g\rightarrow\h$ takes MC elements to MC elements we see that $\MC$ defines a functor on dglas.
\begin{defi}
For a dgla $\g$ and $\xi\in\MC(\g)$ define the differential $d^\xi$ on $\g$ as $d^\xi:=d + \ad(\xi)$. It is straightforward to check that $(d^\xi)^2=0$. We will denote by $\g^\xi$ the dgla whose underlying graded Lie algebra is the same as $\g$ but instead of the differential given by just $d$, the differential is given by $d^\xi$.
\end{defi}

\begin{prop}
Let $\g$ be a dgla and let $\xi\in\g$ be an odd element. Then there is a bijection $\MC(\g^\xi)\rightarrow\MC(\g)$ given by $\eta \mapsto \eta + \xi$.
\end{prop}

\begin{proof}
We calculate that
\begin{align*}
d(\eta + \xi) + \frac{1}{2}[\eta+\xi,\eta+\xi] &= d\xi + \frac{1}{2}[\xi,\xi] + d\eta  +\frac{1}{2}[\eta,\xi] + \frac{1}{2}[\xi,\eta] +\frac{1}{2}[\eta,\eta]
\\
&d\eta + d^\xi\eta + \frac{1}{2}[\eta,\eta]
\end{align*}
and so $\eta\in\MC(\g^\xi)$ if and only if $\eta+\xi\in\MC(\g)$.
\end{proof}

\subsection{The Maurer--Cartan moduli set}
Denote by $\ground[z,dz]$ the free unital differential graded commutative algebra on the generators $z$ and $dz$ with $|z| = 0$, $|dz| = 1$ and $d(z) = dz$.
\begin{defi}
Two elements $\xi,\eta\in\MC(\g)$ are called \emph{homotopic} if there is an element $h\in\MC(\g\otimes\ground[z,dz])$ with $h|_{0} = \xi$ and $h|_{1} = \eta$.
\end{defi}

Homotopy of Maurer--Cartan elements defines a relation that may not be transitive (unless $\CE(\g)$ is cofibrant) so we will consider the transitive closure.

\begin{defi}
We denote by $\MCmod(\g)$ the set of equivalence classes under the transitive closure of the homotopy relation. We call this \emph{the Maurer--Cartan moduli set of $\g$}.
\end{defi}

\subsection{Gauge equivalence}
Let $\g$ be a \emph{pronilpotent} Lie algebra, by which we mean an inverse limit of nilpotent algebras. Recall that for every such Lie algebra there is an associative product $\bullet\co\g\times\g\rightarrow \g$ given by the Baker--Campbell--Hausdorff formula which is functorial (given $f\co\g\rightarrow\h$ then $f(x\bullet y) = f(x)\bullet f(y)$) and for any unital associative algebra $A$ with pronilpotent ideal $I$ it holds for any $a,b\in I$ that $e^a e^b = e^{a\bullet b}$ where $e^a = \sum_{n\geq 0} \frac{a^n}{n!}\in A$ and $\bullet$ is taken with respect to the commutator Lie bracket on $A$. A property of $\bullet$ is that for any $x,y\in\g$ if $[x,y]=0$ then $x\bullet y = x + y$.

Define the group $\exp(\g) = \{e^x : x\in\g \}$ with product defined as $e^x\cdot e^y = e^{x\bullet y}$. The identity is $1=e^0$ and $e^x\cdot e^{-x} = e^{-x}\cdot e^{x} = 1$. It follows from the pronilpotency of $\g$ and the above properties of $\bullet$ that the adjoint representation $y\mapsto \ad(y)$ exponentiates to an action of $\exp(\g)$ on $\g$ given by $e^y\mapsto e^{\ad(y)}$.

Let $\g$ be a pronilpotent dgla and let $\xi\in\MC(\g)$ and $y\in\g^0$. Define the \emph{gauge action} by
\[
e^y\cdot \xi = e^{\ad(y)}\xi + (de^{\ad(y)})y = \xi + \sum_{n=1}^{\infty} \frac{1}{n!}(\ad(y))^{n-1}(\ad(y)\xi-dy).
\]
Then this indeed gives an action of $\exp(\g^0)$ on $\MC(\g)$.

\begin{prop}\label{prop:stabcycles}
Let $\g$ be a pronilpotent dgla. Given $\xi\in\MC(\g)$ then $\exp(\g^0)_\xi = \{ e^y : (d+d^\xi)y = 0 \}$ where $\exp(\g^0)_\xi$ is the stabiliser of $\xi$ by the gauge action.
\end{prop}

\begin{proof}
Let $y\in \g^0$. Since $(d+d^{\xi})y=-(\ad(y)\xi - dy)$ then if $(d+d^{\xi})y=0$ it is clear that $e^y\cdot \xi = \xi$. Conversely since $\g=\lim_{\leftarrow}\g_i$ it is sufficient to show that $\ad(y)\xi - dy=0$ under the image of every $\g\rightarrow\g_i$. The $\g_i$ are nilpotent so for each $\g_i$ there exists some least $N$ such that $(\ad(y))^{N}(\ad(y)\xi-dy)=0$. Then $(\ad(y))^{N-1}(e^y\cdot\xi - \xi) = (\ad(y))^{N-1}(\ad(y)\xi - dy) = 0$ so $\ad(y)\xi-dy=0$ as required.
\end{proof}

\begin{defi}
Let $\g$ be a pronilpotent dgla. Two Maurer--Cartan elements $\xi,\eta\in\MC(\g)$ are called \emph{gauge equivalent} if there is an element $y\in \g^0$ such that $e^y\cdot\xi = \eta$.
\end{defi}

It is natural to also consider the quotient of $\MC(\g)$ by the gauge action. In fact the following important result, due originally to Schlessinger--Stasheff \cite{SS}, tells us the quotient is precisely $\MCmod(\g)$:

\begin{theorem}[Schlessinger--Stasheff theorem]\label{thm:ssthm}
Two Maurer--Cartan elements are gauge equivalent if and only if they are homotopic.
\end{theorem}

\begin{proof}
If $\eta = e^y\cdot\xi$ then $e^{yz}\cdot\xi$ is a homotopy from $\xi$ to $\eta$. The converse is less straightforward and omitted here. Instead see, for example, \cite{CL'}.
\end{proof}

Gauge equivalence is often more convenient to work with than homotopy equivalence.

Let $f\co\g\rightarrow \h$ be a map of pronilpotent dglas and an element $\xi_0\in\MC(\h)$. We wish to understand the moduli spaces of lifts of $\xi_0$ to an $\MC$ element in $\g$. This can be thought of as a general version of certain deformation theory problems.

\begin{rem}
We will not consider here the most general case, when there are obstructions to lifting. In other words we will assume that there exists at least one $\MC$ element $\xi\in \MC(\g)$ such that $f(\xi)=\xi_0$. This assumption can be easily dropped by working with \emph{curved} Lie algebras (for example, see \cite[Chapter 7]{braun}) but we will not need this here.
\end{rem}

\begin{prop}\label{prop:mcfibre}
Let $f\co\g\rightarrow \h$ be a map of dglas. Let $\xi_0\in\MC(\h)$ and $\xi\in\MC(\g)$ with $f(\xi)=\xi_0$. Let $\k^\xi\subset\g$ be the kernel of $f$, endowed with the twisted differential $d^\xi:=d+\ad(\xi)$.
The fibre over $\xi_0\in\MC(\h)$ of $\MC(\g)\rightarrow\MC(\h)$ is isomorphic to $\MC(\k^\xi)$ by the map $\MC(\k^{\xi})\subset\MC(\g^{\xi})\xrightarrow{\eta\mapsto\eta+\xi}\MC(\g)$. In particular it is independent of the choice of $\xi$.
\end{prop}

\begin{proof}
It is elementary to check that an odd element $\eta$ in ${\k^{\xi}}$ is MC if and only if $\eta + \xi\in\MC(\g)$. Moreover, $f(\eta+\xi) = f(\eta) + \xi_0 = \xi_0$ so $\eta+\xi$ indeed belongs to the fibre of $f$ over $\xi_0$.
\end{proof}

\begin{rem}
Note that, despite what the notation may suggest, $\k^{\xi}$ is \emph{not} necessarily obtained by twisting $\k$ by some element of $\k$.
\end{rem}

\begin{prop}\label{prop:gaugeequivoffibres}
Let $f\co\g\twoheadrightarrow \h$ be a surjective map of pronilpotent dglas. The fibres over any two gauge equivalent (and hence homotopy equivalent) elements $\xi_0, \xi_0'\in\MC(\h)$ are isomorphic.
\end{prop}

\begin{proof}
There is a degree zero element $y_0\in\h$ such that $e^{y_0}\cdot \xi_0 = \xi_0'$. Choose elements $\xi,y\in\g$ with $f(\xi)=\xi_0$ and $f(y)=y_0$ so that $\MC(\k^{\xi})$ is the fibre over $\xi_0$ and $\MC(\k^{e^{y}\cdot\xi})$ is the fibre over $\xi_0'$ (since $f(e^y\cdot\xi) = e^{y_0}\cdot\xi_0 = \xi_0'$). Given $\eta\in\MC(\k^{\xi})$ set $g(\eta) = e^{y}\cdot \eta$. Then $g(\eta)\in \MC(\k^{e^{y}\cdot\xi})$ since $f(e^y\cdot\eta + e^y\cdot\xi)= \xi_0'$. This gives an isomorphism $g\co\MC(\k^{\xi})\rightarrow\MC(\k^{e^y\cdot\xi})$ as required.
\end{proof}

We now wish to understand the space of lifts \emph{up to homotopy}, or equivalently the fibre in $\MCmod(\g)$ over $\xi_0\in\MCmod(\h)$. From now on it will be assumed that $f\co\g\twoheadrightarrow\h$ is a surjective map between pronilpotent dglas, $\xi_0\in(\h)$ will be a fixed MC element and $\xi\in \MC(\g)$ will be a given lift of $\xi_0$ so that $f(\xi)=\xi_0$.

There are two natural ways of speaking about equivalence of lifts of $\xi_0$. We could of course say that $\eta,\eta'\in\MC(\k^\xi)$ are equivalent if they are equivalent as Maurer--Cartan elements in the dgla $\k^\xi$. In terms of Sullivan homotopy this means that there is an element $h\in\MC(\k^\xi[z,dz])$ with $h|_0 = \eta$ and $h|_1 = \eta'$. In particular at every value for $z$ it is the case that $h$ is an element of $\MC(\k^\xi)$. In other words this is a homotopy through lifts of $\xi_0$, or more precisely, this homotopy is a lift of the constant homotopy $\xi_0\in \MC(\h[z,dz])$. Therefore, the space of lifts of $\xi_0$ modulo equivalence in this sense is then just $\MCmod(\k^\xi)$ and does not depend on the choice of $\xi$. Note that if $\xi_0'=e^{y_0}\cdot\xi_0$ then given $y\in\g$ with $f(y)=y_0$ and a homotopy $h\in\MC(\k^{\xi}[z,dz])$ with $h|_0 = \eta$ and $h|_1 = \eta'$ then $e^y\cdot h\in\MC(\k^{e^y\cdot\xi}[z,dz])$ is a homotopy from $e^y\cdot\eta$ to $e^y\cdot\eta'$ so by Proposition~\ref{prop:gaugeequivoffibres} gauge equivalent elements have the same space of lifts up to homotopy in this sense. Therefore the space of lifts of up to homotopy in this sense is well defined for a homotopy class $\xi_0\in\MCmod(\h)$.

Alternatively we could say that $\eta,\eta'\in\MC(\k^\xi)$ are equivalent if they are equivalent as Maurer--Cartan elements in $\g$. That is, the elements $\eta+\xi, \eta'+\xi\in\MC(\g)$ are equivalent as Maurer--Cartan elements. In general this will not be the same since two elements may now be homotopic via a homotopy not necessarily through lifts of $\xi_0$. The space of lifts up to homotopy in this sense is just the fibre in $\MCmod(\g)$ over $\xi_0\in\MCmod(\h)$. These two different notions are related as follows.

\begin{theorem}\label{thm:fibration}
Let $f\co\g\twoheadrightarrow \h$ be a surjective map of pronilpotent dglas. Then $\exp(H^{0}(\h^{\xi_0}))$ acts on $\MCmod(\k^\xi)$ and $\MCmod(\k^\xi)/\exp(H^{0}(\h^{\xi_0}))$ is isomorphic to the fibre in $\MCmod(\g)$ over $\xi_0\in\MCmod(\h)$.
\end{theorem}

\begin{proof}
Given $y_0\in\h^0$ such that $(d+d^{\xi_0})y_0=0$ and $y,y'\in\g^0$ with $f(y)=f(y')=y_0$ then for any $\eta\in\MC(\k^\xi)$ let
\[h=e^{y'z}\cdot e^{-yz}\cdot e^{y}\cdot(\eta+\xi)\]
so that $h\in\MC(\g[z,dz])$ is a homotopy from $e^y\cdot(\eta+\xi)$ to $e^{y'}\cdot(\eta+\xi)$. Then $f(h)=e^{y_0}\cdot \xi_0 = \xi_0$ by Proposition~\ref{prop:stabcycles} so $h$ is a homotopy through lifts of $\xi_0$ and so $e^y\cdot(\eta+\xi) - \xi$ and $e^{y'}\cdot(\eta + \xi) - \xi$ are homotopy equivalent as elements in $\MC(\k^\xi)$. Therefore this gives a well defined action $e^{y_0}\star\eta = e^y\cdot(\eta+\xi)-\xi$ on $\MCmod(\k^\xi)$ for cycles in $(\h^{\xi_0})^0$. Furthermore given $x_0\in\h^0$ such that $x_0 = y_0 + (d+d^{\xi_0})b_0$ for some $b_0\in\h^{-1}$ and $b\in\g^{-1}$ with $f(b)=b_0$ set $x=y+(d+d^{\xi})b$ so that $f(x)=x_0$. Set
\[h = e^{y + (d + d^\xi)(bz)}\cdot (\eta+\xi)\]
for $\eta\in\MC(\k^\xi)$ and then $h\in\MC(\g[z,dz])$ so $h$ gives a homotopy from $e^y\cdot(\eta+\xi)$ to $e^x\cdot(\eta+\xi)$. Since $y_0+(d+d^{\xi_0})(b_0z)$ is a cycle in $\h^{\xi_0}[z,dz]$ then again by Proposition~\ref{prop:stabcycles} $f(h) = \xi_0$ so $h$ is a homotopy through lifts of $\xi_0$ and hence $e^{y_0}\star\eta$ and $e^{x_0}\star\eta$ are equivalent as elements in $\MC(\k^\xi)$. Therefore $\star$ descends to a well defined action of $\exp(H^0(h^{\xi_0}))$ on $\MCmod(\k^\xi)$.

That $\MCmod(\k^\xi)/\exp(H^0(\h^{\xi_0}))$ is isomorphic to the fibre in $\MCmod(\g)$ over $\xi_0\in\MCmod(\h)$ follows from how the action of $\exp(H^0(\h^{\xi_0}))$ was defined above together with Proposition~\ref{prop:stabcycles}.
\end{proof}

\begin{rem}\label{rem:conceptualfibration}
There is perhaps a more conceptual, albeit less elementary, way to understand Theorem~\ref{thm:fibration} from a topological viewpoint. Let $\g$ and $\h$ be complete dglas (by which we mean an inverse limit of \emph{finite dimensional} nilpotent dglas). Then a fibration of $\g\twoheadrightarrow \h$ induces a fibration $\MC_\bullet(\g)\twoheadrightarrow \MC_\bullet(\h)$ where $\MC_\bullet$ is the \emph{Maurer--Cartan simplicial set}. Indeed one way to see this is as follows. The model category of cdgas is almost a simplicial model category, for example using the results of \cite{hin}, in particular it satisfies the corner axiom: Given a cofibration $i\co A\rightarrow B$ and a fibration $p\co X\rightarrow Y$ the induced map
\[(i^*,p_*)\co \Hom(B,X)_\bullet \rightarrow \Hom(A,X)_\bullet\times_{\Hom(A,Y)_\bullet}\Hom(B,Y)_\bullet\]
is a fibration of simplicial sets. Next, a fibration $g\twoheadrightarrow \h$ induces a cofibration $\CE(\h)\rightarrowtail\CE(\g)$ of cdgas ($\g\mapsto \CE(\g)$ is a contravariant Quillen functor) and hence by the corner axiom induces a fibration $\Hom(\CE(\g),\ground)_\bullet\twoheadrightarrow \Hom(\CE(\h),\ground)_\bullet$. But for any complete dgla $\g$ then $\MC_\bullet(\g)$ is precisely $\Hom(\CE(\g),\ground)_\bullet$ and this is the required fibration of Maurer--Cartan simplicial sets.

Now a version of Theorem~\ref{thm:fibration} can be obtained, at least for $\g$ and $\h$ complete dglas over $\mathbb{Q}$ of finite type, by considering the long exact sequence in homotopy of this fibration together with the facts that $\pi_0\MC_\bullet(\g)=\MCmod(\g)$ and $H^0(\g^\xi)$ is $\pi_1$ of the connected component of $\MC_\bullet(\g)$ containing $\xi$. In particular, from this point of view the standard picture in Figure~\ref{fig:mcfibre} now becomes quite enlightening to keep in mind.
\end{rem}

\begin{figure}[ht!]
\centering
{\footnotesize
\[
\begin{xy}
*\xybox{+(0,7),{\ellipse va(210),_,va(-30){-}};
+(0,14),{\ellipse va(210),_,va(-30){-}};
+(0,14),{\ellipse va(210),_,va(-30){-}};}
!UC+(0,5)*\xybox{\ellipse(8,3){-}}="centery"
-(0,45)*\xybox{\ellipse(8,3){-}}+(0,5);p+(0,8);**\dir{-}*\dir2{>}
-(4,10)*{\xi_0}="xi",
+(0,2.4)*\dir{*};
p+(0,19)*\dir{*}="1",
+(0,7)*\dir{*}="2",
+(0,7)*\dir{*}="3",
+(0,12)*\dir{*}="4",
+(0,7),
**\dir{.},
"xi"+(24,5)*{\MC_\bullet(\h)}
+(0,30)*{\MC_\bullet(\g)}
-(40,0)*{\MC_\bullet(\k^{\xi})}+(6,0)="k";
"1"-(1.5,0.5)**\crv{~**\dir{.} "k"-(0,15)}?(1)*\dir{>},
"2"-(3.5,-0.5)**\crv{~**\dir{.} "k"-(-2,5)}?(1)*\dir{>},
"3"-(3.5,-0.7)**\crv{~**\dir{.} "k"+(2,2)}?(1)*\dir{>},
"4"-(1.5,0.5)**\crv{~**\dir{.} "k"+(0,10)}?(1)*\dir{>},
\end{xy}
\]}
\caption{In this standard picture of a fibration $\MCmod(\h)$ has one element, $\MCmod(\g)$ has two elements and $\MCmod(\k^{\xi})$ has four elements, although the fibre in $\MCmod(\g)$ over $\xi_0$ has two elements.}
\label{fig:mcfibre}
\end{figure}
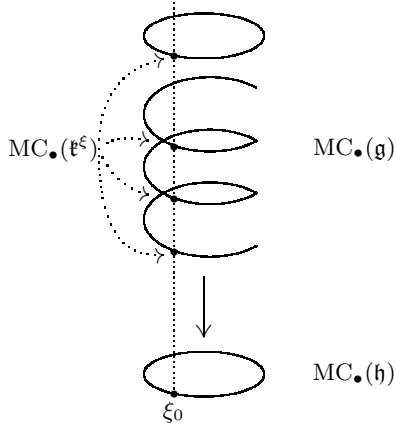

\section{Minimal models of operadic algebras}\label{app:operads}
In this section we give a treatment of minimal models of algebras over operads and cyclic operads.
Our references for the theory of dg operads are \cite{GiK, GeK} and we will be freely using results and terminology from these sources.

Let $\O$ be a dg operad which we assume to be non-unital and admissible; i.e.~$\O(n)=0, n=0,1$ and all spaces $\O(n)$ are finite dimensional for $n=2,3,\dots$. We will denote by $\B\O$ the cobar construction of $\O$. If $\O$ is quadratic, we denote by $\O^!$ the quadratic dual to $\O$ and by $\Lambda\O$ the operadic suspension of $\O$. There is a natural map $p_\O\co\B\Lambda\O^!\twoheadrightarrow \O$ which is a quasi-isomorphism in the case when $\O$ is Koszul. If $\O$ is a cyclic operad then $\B\O$ and $\Lambda\O^!$ will be anti-cyclic operads and vice-versa.  In the case that $\O$ is cyclic then $\Lambda \O$ is anti-cyclic and vice-versa; $\Lambda$ is an equivalence of categories between cyclic and anti-cyclic operads. The tensor product of two operads $\O_1$ and $\O_2$ is defined component-wise so that $\O_1\otimes \O_2(n)=\O_1(n)\otimes\O_2(n)$. The tensor product of two cyclic operads is a cyclic operad, the tensor product of two anti-cyclic operads is a cyclic operad and the tensor product of an anti-cyclic and a cyclic operad is an anti-cyclic operad.

The category of reduced operads has the structure of a closed model category, cf.~\cite{BeM} where weak equivalences are component-wise quasi-isomorphisms and fibration are component-wise surjections. Similarly, there is closed model category structure on the category of reduced cyclic operads obtained in \cite{luc}.

For a dg vector space $V$ we denote by $\E(V)$ the \emph{endomorphism operad} of $V$, so that $\E(V)(n)=\Hom(V^{\otimes n}, V), n=2,3,\dots$. In the case when $V$ has a non-degenerate inner product, the operad $\E(V)$ is cyclic. If the non-degenerate inner product on $V$ is even then we define the \emph{cyclic endomorphism operad} of $V$ to be the cyclic operad $\Ecyc(V)=\E(V)$. If the non-degenerate inner product on $V$ is odd then we define the \emph{anti-cyclic endomorphism operad} of $V$ to be the anti-cyclic operad $\Ecyc(V)=\Lambda \E(V)$. An algebra over an operad $\O$ is a dg vector space $V$ and an operad map $\O\to\E(V)$. An algebra over a cyclic operad $\O$ is a dg vector space $V$ with a non-degenerate even inner product and a cyclic operad map $\O\to\Ecyc(V)$. Similarly, an algebra over an anti-cyclic operad $\O$ is a dg vector space $V$ with a non-degenerate odd inner product and an anti-cyclic operad map $\O\to \Ecyc(V)$.
The operad governing commutative algebras is denoted by $\Com$; thus $\Com(n)=\ground$ for $n>1$ and $\Com(n)=0$ for $n=0,1$. The algebras of the cofibrant operad $\B\Com$ are $L_\infty$~algebras (introduced in Section~\ref{sec:linfty} using a different approach). The operad $\Com$ is cyclic and algebras over the cofibrant anti-cyclic operad $\B\Com$ are odd cyclic $L_\infty$~algebras.

\subsection{Weakly non-degenerate bilinear forms and endomorphism operads}
We will need a certain extension of the notion of a cyclic algebra to the case of a possibly degenerate inner product. Let $V$ be a dg vector space together with a symmetric bilinear form $\langle,\rangle$, possibly degenerate and $\O$ be a (anti\nobreakdash-)cyclic operad. In this situation it still makes sense to consider a (anti\nobreakdash-)cyclic $\O$--algebra structure on $V$ by imposing a suitable (anti\nobreakdash-)cyclic invariance condition on $\langle, \rangle$, cf, \cite[Section 4]{GeK}. In the case when $\langle, \rangle$ is non-degenerate, this notion is equivalent to a map of (anti\nobreakdash-)cyclic operads $\O\to\Ecyc(V)$ as above. In the degenerate case the endomorphism operad $\E(V)$ is not cyclic, so the definition of the (anti\nobreakdash-)cyclic endomorphism operad given above no longer makes sense. However we will see that there still exists a suitable analogue of the (anti\nobreakdash-)cyclic endomorphism operad under the condition that $\langle, \rangle$ is \emph{weakly non-degenerate}, i.e.~that it induces an injective map $V\hookrightarrow V^*$ (note that in the case $\dim V$ is finite this notion coincides with the usual notion of non-degeneracy).

The map $V\to V^*$ induced by $\langle,\rangle$ gives rise to maps \[\alpha_n\co\Hom(V^{\otimes n},V)\cong V^{*\otimes n}\otimes V\to V^{*\otimes n+1}.\]
Note that in the case when $\langle,\rangle$ is weakly non-degenerate the maps $\alpha_n$ are injective. To alleviate notation, we will suppress the subscript $n$ from now on.
\begin{defi}\label{def:cyclicendomorphism}
Let $V$ be a dg vector space with a weakly non-degenerate inner product $\langle,\rangle$. Let $\E_c(V)(n)\subset\E(V)(n)$ be the subspace consisting of such multi-linear maps $f\in \Hom(V^{\otimes n},V)$ that for any $\sigma\in S_{n+1}$ the element $\sigma(\alpha(f))$ lies in the image of $\alpha$. The operad structure maps in $\E_c(V)$ are determined by requiring that it be a suboperad of $\E(V)$ and the action of $S_{n+1}$ on $\E_c(V)$ is given by the inclusion $\E_c(V)(n)\hookrightarrow V^{*\otimes n+1}$.
\end{defi}
\begin{prop}
Definition~\ref{def:cyclicendomorphism} makes $\E_c(V)$ into a cyclic operad.
\end{prop}
\begin{proof}
Given an element $h\in V^{*\otimes n+1}$ let $h\mapsto h^*$ be the action of the cycle $(1\, 2\, \dots \, n+1)$ on $V^{*\otimes n+1}$. Given an element $f\in \E_c(V)(n)$ denote similarly by $f\mapsto f^*$ the result of the action of this cycle on $E_c(V)(n)$. To show that $E_c(V)$ is a cyclic operad we must show the following two conditions are satisfied for all $f\in \E_c(V)(n)$ and $g\in E_c(V)(m)$:
\begin{enumerate}
\item It holds that $f\circ_i g \in \E_c(V)(n+m-1), i=1,\dots, m$, or equivalently $(\alpha(f\circ_i g))^*\in \im(\alpha)$.
\item It holds that $(f \circ_n g)^* = g^* \circ_1 f^*$.
\end{enumerate}
 We will show that $(\alpha(f\circ_n g))^* = \alpha(g^*\circ_1 f^*)$. This identity clearly implies condition (1) above. Furthermore, since $\alpha$ is injective, it also then implies condition (2) above. Let $v_1,\dots, v_{n+m}\in V$. For simplicity we assume that the elements $f,g,v_i$ are even; the proof in the general case is slightly messier because of the signs. Then we have
 \begin{align*}
 (\alpha(f\circ_n g))^*(v_1,\dots, v_{n+m}) & = \alpha(f\circ_n g)(v_2,\dots, v_{n+m}, v_1) \\
 &= \langle (f\circ_n g)(v_2,\dots , v_{n+m}), v_1 \rangle\\
 &= \langle f(v_2,\dots v_n, g(v_{n+1},\dots, v_{n+m})), v_1 \rangle \\
 &= \langle f^*(v_1,\dots, v_n), g(v_{n+1},\dots, v_{n+m}) \rangle\\
 & = \langle g(v_{n+1},\dots, v_{n+m}),f^*(v_1,\dots, v_n) \rangle \\
 & = \langle g^*(f^*(v_1,\dots, v_n), v_{n+1},\dots, v_{n+m-1}), v_{n+m} \rangle\\
 & = \alpha(g^* \circ_1 f^*)(v_1,\dots , v_{n+m})
 \end{align*}
 as required.
\end{proof}
\begin{rem}
In the case when $V$ is finite dimensional, the operad $\E_c(V)$ is clearly isomorphic to $\E(V)$. In the infinite dimensional case it is strictly smaller; e.g.~if $V$ has an orthonormal basis, then $E(V)(1)$ can be represented as the space of matrices having columns of finite length whereas $\E_c(V)$ consists of matrices having both columns and rows of finite length. In that case the action of $S_2$ is given by transposition of matrices.
\end{rem}

Let $V$ be a dg vector space with a weakly non-degenerate bilinear form so that $E_c(V)$ is a cyclic operad. As before, if the bilinear form is even then we define the \emph{cyclic endomorphism operad} of $V$ to be the cyclic operad $\Ecyc(V)=\E_c(V)$. If the bilinear form is odd then we define the \emph{anti-cyclic endomorphism operad} of $V$ to be the anti-cyclic operad $\Ecyc(V)=\Lambda \E_c(V)$. Let $\O$ be a (anti\nobreakdash-)cyclic operad. It is now straightforward to see that an (anti\nobreakdash-)cyclic $\O$--algebra on $V$ in the sense of \cite[Section 4]{GeK} is the same as a map of (anti\nobreakdash-)cyclic operads $\O\to \Ecyc(V)$.

From now on, by an algebra over a cyclic operad $\O$ we will mean a dg vector space $V$ with a \emph{weakly} non-degenerate even bilinear form and a cyclic operad map $\O\to \Ecyc(V)$. Similarly an algebra over an anti-cyclic operad $\O$ will from now on mean a dg vector space $V$ with a \emph{weakly} non-degenerate odd bilinear form and an anti-cyclic operad map $\O\to \Ecyc(V)$.

\begin{rem}
Another possible definition of a cyclic endomorphism operad for an arbitrary symmetric bilinear form $\langle,\rangle$ is given as a collection of dg vector spaces $\{V^{\otimes n}\}, n=1,2,\ldots$. The operad structure is then defined via the contracting of tensors using $\langle,\rangle$. However, if we had defined the notion of an $\O$--algebra using this definition for the cyclic endomorphism operad, it is not clear that there are any interesting examples of $\O$--algebras when $V$ is not finite dimensional (and in the finite dimensional case it agrees with our definition). Therefore, while it is likely that one can obtain a minimal model theorem for this version of a cyclic endomorphism operad, the utility of this result is unclear in the absence of interesting examples.
\end{rem}

\subsection{Homotopy equivalence and minimal models}
We will now introduce a notion of \emph{homotopy equivalence} of operadic algebras and prove a version of the `minimal model theorem' for algebras over operads and (anti\nobreakdash-)cyclic operads.

In order to define our notion of homotopy equivalence of two operadic algebras we will need to understand when maps between dg vector spaces induce maps between endomorphism operads. To this end, recall the notion of a strong deformation retract.

\begin{defi}\label{def:sdr}
A \emph{strong deformation retraction} (SDR) from vector space $V$ to a vector space $B$ is a pair of operators
\[i\co B\hookrightarrow V\quad \text{and}\quad p\co B\twoheadrightarrow V\]
of even degree and an operator $s\co V\to V$ of odd degree, such that:
\newcounter{sdrconditions}
\begin{enumerate}
\item $di = id$ and $dp=pd$
\item $pi= \id_B$
\item $ds + sd = \id_V - ip$
\item $si = 0$ and $ps = 0$
\item $s^2 = 0$
\setcounter{sdrconditions}{\value{enumi}}
\end{enumerate}
If $V$ is equipped with a bilinear form $\langle, \rangle$ then it is also required that $B$ have a bilinear form $\langle,\rangle$ such that:
\begin{enumerate}
\setcounter{enumi}{\value{sdrconditions}}
\item $\langle ix,iy \rangle = \langle x,y\rangle$
\item $\Ker p \perp \im i$
\item $\langle sx,y \rangle = (-1)^{|x|}\langle x, sy \rangle$
\setcounter{sdrconditions}{\value{enumi}}
\end{enumerate}
\end{defi}
The conditions (4) and (5) are the so-called \emph{side conditions} and they are not always included in the definition of SDR data. The following result shows that they could always be imposed at no cost.
\begin{lem}
Let $V$ and $B$ be dg vector spaces and $(i,p,s)$ be maps satisfying conditions (1), (2) and (3) above and, in the presence of bilinear forms, conditions (6), (7) and (8). Then
the map $s$ can be replaced by $s^\prime$ in such a way so that $(i,p,s^\prime)$ is an SDR.
\end{lem}
\begin{proof}
If condition (4) is not satisfied, replace $s$ with $\tilde{s} = (ds+sd)s(ds+sd)$ to obtain a triple $(i,p,\tilde{s})$ additionally satisfying condition (4): $\tilde{s}i = 0$ and $p\tilde{s} = 0$. If $(i,p,\tilde{s})$ satisfies condition (4) but not condition (5) then replace $\tilde{s}$ with $\hat{s} = \tilde{s}d\tilde{s}$ to obtain a triple $(i,p,\hat{s})$ additionally satisfying both (4) and (5).
\end{proof}

\begin{rem}
One can show that the notion a strong deformation retract is equivalent to that of an abstract Hodge decomposition, cf.~\cite{CL', CL''}.
\end{rem}

The following proposition shows that an SDR from $V$ to $B$ is precisely the data required to induce a quasi-isomorphism of operads $\E(B)\hookrightarrow \E(V)$ or (anti\nobreakdash-)cyclic operads $\Ecyc(B)\hookrightarrow \Ecyc(V)$.

\begin{prop}\label{prop:operadquasiiso}
Let $B, V$ be dg vector spaces. Then a pair $i\co B\hookrightarrow V$ and $p\co V\twoheadrightarrow B$ satisfying conditions (1) and (2) of a strong deformation retract induces a map of operads $\E(B)\hookrightarrow \E(V)$ which is a quasi-isomorphism if and only if there is $s\co V\to V$ such that $(i,p,s)$ is a strong deformation retract.

If $B$ and $V$ have weakly non-degenerate bilinear forms and the pair $i,p$ also satisfy conditions (6) and (7) of a strong deformation retract then there is an induced map of (anti\nobreakdash-)cyclic operads $\Ecyc(B)\hookrightarrow \Ecyc(V)$ which is a quasi-isomorphism if and only if there is $s\co V\to V$ such that $(i,p,s)$ is a strong deformation retract.
\end{prop}

\begin{proof}
Define maps $\E(B)(n)\hookrightarrow\E(V)(n), n=1,2,\dots$ by $f\mapsto ifp^{\otimes n}$; these clearly constitute a map of operads $\E(B)\hookrightarrow\E(V)$. Let us show that in the presence of bilinear forms this restricts to a map of cyclic operads $\E_c(B)\to \E_c(V)$ which will immediately imply the existence of a map of (anti\nobreakdash-)cyclic operads $\Ecyc(B)\hookrightarrow \Ecyc(V)$ as required.
Let $f\in\E_c(B)(n)$ and $f^*$ be the result of applying the cycle $(1\,2\,\dots\, n)$ to $f$.
It suffices to show that
\begin{equation}\label{eq:cyclic}
\langle if (p(x_1),\dots, p(x_n)),x_{n+1}\rangle=\pm\langle if^*(p(x_2),\dots,p(x_{n+1})),x_1\rangle
\end{equation}
where $x_1,\dots,x_n\in V$ and the $\pm$ is worked out from the Koszul sign rule. For clarity of presentation we will assume that $f$ and $x_i$s are all even, so in that case the sign $\pm$ disappears; the proof in the general case is the same, albeit messier. By our assumption there is an orthogonal decomposition $V\cong \Ker p\oplus\im i$. If $x_{1}\in \Ker p$ or $x_{n+1}\in\Ker p$ then both sides of equation \eqref{eq:cyclic} vanish. Therefore assume that $x_1, x_{n+1}\in\im i$ so that $x_1=i(w)$ and $x_{n+1}=i(z)$ for some $w,z\in B$. We have
\begin{align*}
\langle if (p(x_1),\dots, p(x_n)),x_{n+1}\rangle&=\langle if(pi(w),\dots, p(x_n)),i(z)\rangle\\
&=\langle f(w,\dots, p(x_n)),z\rangle\\
&=\langle f^*(p(x_2),\dots, p(x_n),z),w\rangle\\
&=\langle if^*(p(x_2),\dots, p(x_n),pi(z)),i(w)\rangle\\
&=\langle if^*(p(x_2),\dots,p(x_{n+1})),x_1\rangle
\end{align*}
as desired.

Assume there exists $s\co V\to V$ such that $(i,p,s)$ is a strong deformation retract. To show the map of operads $\E(B)\hookrightarrow \E(V)$ is a quasi-isomorphism we must show that the maps $\E(B)(n) \hookrightarrow \E(V)(n)$ given by $f\mapsto i f p^{\otimes n}$ are isomorphisms on homology. Since $p i = \id_{B}$, the maps $\E(V)(n)\twoheadrightarrow \E(B)(n)$ given by $f\mapsto p f i^{\otimes n}$ are left inverses. To show that these maps induce right inverses on homology, let $f\in \E(V)(n)$ be a cycle, then it is sufficient to show that $ip f(ip)^{\otimes n}$ is homologous to $f$ as an element in the chain complex $\E(V)(n)$. But $d(sf(ip)^{\otimes n}) = f(ip)^{\otimes n} - (ip)f(ip)^{\otimes n}$ so $ipf(ip)^{\otimes n}$ and $f(ip)^{\otimes n}$ are homologous as cycles in $\E(V)(n)$. Similarly, for $k< n$ we have \[d(f\circ (\id^{\otimes k} \otimes s\otimes (ip)^{\otimes n-k})) = f\circ (\id^{\otimes k+1}\otimes (ip)^{\otimes n-k-1}) - f\circ (\id^{\otimes k}\otimes (ip)^{\otimes n-k})\] so, by induction, $ipf(ip)^{\otimes n}$ is homologous to $f$ as required. The (anti\nobreakdash-)cyclic case is the same, except it is necessary that elements such as $sf(ip)^{\otimes n}$ are in $\Ecyc(V)(n)$, but this is true by condition (8) of a strong deformation retract.

Conversely, assume that the map of operads $\E(B)\hookrightarrow \E(V)$ is a quasi-isomorphism. Then in particular $ip\in E(V)(1)$ is homologous to the identity, in other words there is some $s\co V\to V$ such that $ds+sd = \id_V - ip$ as required. Similarly, if $V$ and $B$ have bilinear forms and the map of operads $\Ecyc(B)\hookrightarrow \Ecyc(V)$ is a quasi-isomorphism then $ip\in \Ecyc(V)(1)$ is homologous to the identity, in other words there is some $s\in \Ecyc(V)(1)$ such that $ds+sd = \id_V - ip$. By replacing $s$ with $\tilde{s} = (s+s^*)/2$, where $s^*$ is the action of the cycle $(1\, 2)$ on $s\in\Ecyc(V)(1)$, then $d\tilde{s}+\tilde{s} = \id_V - ip$ and $\langle \tilde{s}x,y\rangle = (-1)^{|x|}\langle x,sy \rangle$ as required.
\end{proof}

\begin{defi}\label{def:homotopyequiv}\
\begin{itemize}
\item
Let $\P$ be a cofibrant operad and $\P\to\E(A)$ be a $\P$--algebra. Then a $\P$--algebra $\P\to\E(B)$ is a \emph{model} for $A$ if there is an SDR from $A$ to $B$ such that the following diagram of operads is commutative up to homotopy:
\[
\xymatrix{\P\ar[r]\ar[dr]&\E(A)\\&\E(B)\ar@{^{(}->}[u]}
\]
\item
Similarly, let $\P$ be a cofibrant (anti\nobreakdash-)cyclic operad and $\P\to\Ecyc(A)$ be a $\P$--algebra. Then a $\P$--algebra $\P\to\Ecyc(B)$ is a \emph{model} for $A$ if there is an SDR from $A$ to $B$ such that the following diagram of (anti\nobreakdash-)cyclic operads is commutative up to homotopy:
\[
\xymatrix{\P\ar[r]\ar[dr]&\Ecyc(A)\\&\Ecyc(B)\ar@{^{(}->}[u]}
\]
\item
Two $\P$--algebras $A,A^\prime$ are homotopy equivalent if there exists a $\P$--algebra $B$ that is a model for both $A$ and $A^\prime$.
\end{itemize}
\end{defi}

\begin{defi}
Let $B$ be a model for a $\P$--algebra $A$. If $B$ has vanishing differential it is called a \emph{minimal model} for $A$.
\end{defi}

\begin{theorem}[Minimal model theorem]\label{thm:minmod}\
\begin{enumerate}
\item Let $\P$ be a cofibrant operad and $A$ a $\P$--algebra. Then there exists a minimal model for $A$ which is unique up to homotopy equivalence.
\item Let $\P$ be a cofibrant (anti\nobreakdash-)cyclic operad and $A$ a $\P$--algebra for which the induced form on $H(A)$ is non-degenerate. If $A$ admits an SDR onto $H(A)$ then there exists a minimal model for $A$ which is unique up to homotopy equivalence.
\end{enumerate}
\end{theorem}

\begin{proof}
We first prove (1). Let $C$ be a complement to the space $\im d$ as a subspace of $\Ker d\subset A$, so that $\Ker d \cong \im d \oplus C$. Then $C\cong H(A)$ and so the inclusion $i\co C\cong H(A)\hookrightarrow A$ is a quasi-isomorphism since it is the identity on homology. Therefore there is a splitting $p\co A\twoheadrightarrow H(A)$ so that $pi=\id_{H(A)}$ and $ip$ induces the identity on homology, so is chain homotopic to $\id_A$, in other words there is a strong deformation retraction $(i,p,s)$ of $A$ onto $H(A)$. This induces a quasi-isomorphism of operads $\E(H(A))\hookrightarrow A$, in other words this map is invertible in the homotopy category of operads. Therefore there is a map $\P\to \E(H(A))$, unique up to homotopy, making $H(A)$ into a model for $A$.

We need to show that this construction does not depend, up to homotopy, on the choice of the complement to $\im d$ and the splitting $p$. Let $(i_0,p_0,s_0)$ and $(i_1,p_1,s_1)$ be two strong deformation retractions of $A$ onto $H(A)$ arising in this way. It is sufficient to show that the maps of operads $\E(H(A))\hookrightarrow \E(A)$ induced by these SDRs are homotopic as maps of operads. Note that, by construction, $p_1i_0 = p_0i_1 = \id_{H(A)}$. Define $i\co H(A)[z,dz] \to A [z,dz]$ by setting
 \begin{equation}\label{eq:i}
 i(x) = i_0(x) + (i_1-i_0)(x)z - s_0i_1(x)dz
 \end{equation}
 for $x\in H(A)$ and extending $\ground[z,dz]$--linearly Define $p\co A[z,dz] \to H(A)[z,dz]$ by setting $p(x) =p_0(x)$ for $x\in A$ and extending $\ground[z,dz]$--linearly. Then $pi = \id_{H(A)}$. Therefore, there is a map of operads $\E(H(A))\hookrightarrow \E(H(A))[z,dz]\to E(A)[z,dz]$. Moreover, upon setting $z=0$ we recover the map induced by the pair $(i_0,p_0)$ and upon setting $z=1$ we recover the map induced by the pair $(i_1,p_0)$, therefore these two maps are homotopic as maps of operads. By a similar argument, the maps of operads induced by $(i_1,p_0)$ and $(i_1,p_1)$ are also homotopic as maps of operads. This completes the proof of (1). As an aside, note that the pair $(i_1,p_0)$ could be viewed as part of an SDR and the pair $(i,p)$ is thus a homotopy between the two SDRs corresponding to $(i_0,p_0)$ and $(i_1,p_0)$.

We will now prove (2). By assumption, $A$ admits an SDR $(i_0,p_0,s_0)$ onto $H(A)$, without loss of generality we assume that $i_0\co H(A)\to A$ induces the identity map on homology. Recall from Proposition~\ref{prop:operadquasiiso} that this induces a map of (anti\nobreakdash-)cyclic operads $\Ecyc(H(A))\hookrightarrow E(A)$ which is a quasi-isomorphism. Therefore, as for part (1), this map is invertible in the homotopy category of (anti\nobreakdash-)cyclic operads and so there is a map $\P\to \Ecyc(H(A))$, unique up to homotopy, making $H(A)$ into a model $A$.

We need to show that this construction does not depend, up to homotopy, on the choice of an SDR. Let $(i_0, p_0, s_0)$ and $(i_1,p_1,s_1)$ be two SDRs, such that $i_0,i_1\co H(A)\to A$ induce the identity map on homology. As before it is sufficient to show that the maps of (anti\nobreakdash-)cyclic operads $\Ecyc(H(A))\hookrightarrow \Ecyc(A)$ induced by these SDRs are homotopic as maps of (anti\nobreakdash-)cyclic operads. Note that $H(A)[z,dz]$ and $A[z,dz]$ have $\ground[z,dz]$--bilinear forms defined by
 \[
 \langle x p, yq \rangle = (-1)^{\pi + |p||y|}\langle x, y \rangle pq
 \]
 where $x,y\in H(A)$ or $x,y\in A$, $p,q\in \ground[z,dz]$ and $\pi$ is the parity of the inner products on $A$ and $H(A)$ cf.~for example, \cite{Man}, for the notion of a superalgebra valued bilinear form.
 Define an inclusion $i\co H(A)[z,dz] \to A [z,dz]$ as before, by the same formula \eqref{eq:i}.
 \begin{lem}
 The map $i\co H(A)\to A[z,dz]$ preserves the $\ground[z,dz]$--bilinear forms.
 \end{lem}
 \begin{proof}
 Let $x,y\in A$. Then we have:
 \begin{align*}
 \langle i(x), i(y) \rangle &= \langle i_0(x), i_0(y) \rangle\\
 & + \left ( \langle i_0(x), i_1(y) \rangle  - \langle i_0(x), i_0(y) \rangle  + \langle i_1(x),i_0(y) \rangle - \langle i_0(x), i_0(y) \rangle \right ) z\\
 & + \left ( \langle i_1(x), i_1(y) \rangle - \langle i_1(x), i_0(y) \rangle - \langle i_0(x), i_1(y) \rangle + \langle i_0(x), i_0(y) \rangle \right ) z^2\\
 & + \left ( \langle i_0(x), s_0i_1(y) \rangle - \langle s_0i_1(x), i_0(y) \rangle \right )dz\\
 & + \left ( \langle i_0(x),s_0i_1(y) \rangle + (-1)^{\pi + |y|}\langle s_0i_1(x), i_0(y) \rangle\right )zdz\\
  & - \left ( \langle i_1(x), s_0i_1(y)\rangle + (-1)^{\pi + |y|}\langle s_0i_1(x), i_1(y) \rangle \right )zdz
 \end{align*}
 Since $i_0$ and $i_1$ induce the same maps on homology, there is $c\in A$ such that $i_1(y) = i_0(y) + dc$. Therefore \[\langle i_0(x), i_1(y) \rangle = \langle i_0(x), i_0(y) \rangle + \langle i_0(x), dc \rangle = \langle x, y \rangle\] since $i_0(x)$ is a cycle. Similarly, $\langle i_1(x), i_0(y) \rangle = \langle x, y \rangle$. It follows that the coefficients of $z$ and $z^2$ in $\langle i(x), i(y)\rangle $ vanish. By conditions (4) and (8) of Definition~\ref{def:sdr}, the coefficient of $dz$ also vanishes. For the same reason, the first two coefficients of $zdz$ above also vanish. It remains to show
 \[\langle i_1(x), s_0i_1(y)\rangle + (-1)^{\pi + |y|}\langle s_0i_1(x), i_1(y) \rangle=\langle i_1(x), s_0i_1(y)\rangle + (-1)^{\pi + |y|+|x|}\langle i_1(x), s_0i_1(y) \rangle\]
 is zero. If $|y| + |x| - 1$ is of opposite parity to $\pi$ then both these terms vanish separately. Otherwise, they cancel each other.
\end{proof}

 Note that the given $\ground[z,dz]$--bilinear form on $H(A)[z,dz]$ is non-degenerate in the sense that it induces a $\ground[z,dz]$--linear isomorphism $H(A)[z,dz]\to H(A)^*[z,dz]$. Therefore there is an orthogonal decomposition $A[z,dz]\cong \im i \oplus (\im i)^{\perp}$ and we define $p\co A[z,dz]\to H(A)[z,dz]$ to be the corresponding $\ground[z,dz]$--linear orthogonal projection onto $\im i$. By uniqueness of orthogonal decompositions, $p|_{z=0}= p_0$ and $p|_{z=1}=p_1$. The pair $(i,p)$ yields a map of (anti\nobreakdash-)cyclic operads $\Ecyc(H(A))\hookrightarrow \Ecyc(H(A))[z,dz]\to \E(A)[z,dz]$ and furthermore at $z=0$ we recover the map induced by the pair $(i_0,p_0)$ and at $z=1$ we recover the map induced by the pair $(i_1,p_1)$ and so they are homotopic as maps of (anti\nobreakdash-)cyclic operads as required.
 \end{proof}

\begin{rem}
If a space $A$ with a bilinear form is finite-dimensional, an SDR onto $H(A)$ always exists. In infinite-dimensional cases it often exists too. For example, if $A=\Omega(M)$, the de~Rham algebra on a smooth oriented compact manifold, then the geometric Hodge decomposition on $M$ gives, essentially, an SDR onto $H(A)$, see \cite[Example 2.9 (2)]{CL''}.
\end{rem}

\begin{rem}
Theorem~\ref{thm:minmod} (1) is by now rather well known and we give a proof here mainly for comparison with part (2). The \emph{decomposition theorem} for cyclic algebras was proved in \cite{CL''} from which the existence of cyclic minimal models was deduced, however it does not immediately follow from this that cyclic minimal models are unique up to homotopy equivalence, as $\P$--algebra structures on the homology.

The proof of uniqueness of minimal models given here is slightly different than that usually found in the literature; more commonly one defines the notion of infinity quasi-isomorphism of $\P$--algebras and observes that a minimal model of $A$ is infinity quasi-isomorphic to $A$. Then it follows that any two minimal models are infinity isomorphic (and hence homotopy equivalent) after observing that infinity quasi-isomorphisms are always invertible, in an appropriate sense. However, we have taken a different approach since this argument does not work in part (2), the cyclic case, because cyclic infinity quasi-isomorphisms are not necessarily invertible.
 \end{rem}

\begin{cor}\
\begin{enumerate}
\item Let $\P$ be a cofibrant operad and $A$ and $A'$ be $\P$--algebras. Then $A$ and $A'$ are homotopy equivalent if and only if they have homotopy equivalent minimal models.
\item Let $\P$ be a cofibrant (anti\nobreakdash-)cyclic operad and $A$ and $A'$ be $\P$--algebras admitting SDRs onto their homology and for which the induced forms on homology are non-degenerate. Then $A$ and $A'$ are homotopy equivalent if and only if they have homotopy equivalent minimal models.
\end{enumerate}
\end{cor}

\begin{proof}
Let $B$ be the minimal model of $A$. If $B$ is homotopy equivalent to the minimal model of $A'$ then $B$ is clearly also a model for $A'$.

Conversely let $B$ be a model, not necessarily minimal, for $A$ and $A'$. Then the minimal model of $B$ is also a model for $A$ and $A'$ so, by uniqueness up to homotopy of minimal models, is homotopy equivalent to the minimal models of both $A$ and $A'$.
\end{proof}

\begin{rem}\label{rem:inftyhomotopyequiv}
Two minimal (cyclic) $L_\infty$~algebras are homotopy equivalent if and only if they are (cyclic) $L_\infty$~isomorphic (by, essentially, Theorem~\ref{thm:ssthm} together with Remark~\ref{rem:linftygaugeiso}).
\end{rem}

\section{Tensor products of operadic algebras}\label{app:tensor}
In this section we explain how to form tensor products of operadic algebras, generalizing the tensor products of $L_\infty$~algebras and cdgas, discussed in Section~\ref{sec:tensors}.

For any operad $\O$ there exists a canonical map \begin{equation}\label{eq:tensor}\phi_\O\co\B\Com\to\O\otimes \B\O\end{equation} cf.~\cite[Proposition 3.2.18]{GiK}. If $\O$ is cyclic or anti-cyclic, then \eqref{eq:tensor} is a map of anti-cyclic operads. Further, \eqref{eq:tensor} can be used to define the structure of an (odd cyclic) $L_\infty$~algebra on a tensor product of an $\O$--algebra and a $\B\O$--algebra. For example, the tensor product of an odd cyclic $L_\infty$~algebra and a cyclic cdga has a structure of an odd cyclic $L_\infty$~algebra, a fact that was used in Section~\ref{sec:tensors}.

\begin{lem}
Let $\O$ be a Koszul operad. There exists a map of operads $\Phi_{\O}\co\B\Com\to\B\O\otimes\B\Lambda\O^!$, unique up to homotopy, making commutative the following diagram:
\[
\xymatrix{&\B\Lambda\O^!\otimes\B\O\ar^{p_{\O}\otimes \id}@{->>}[d]
\\
\B\Com\ar^{\Phi_{\O}}[ur]\ar_{\phi_{\O}}[r]&\O\otimes\B\O}\]
If $\O$ is a cyclic or anti-cyclic operad then $\Phi_\O$ can be chosen to be a map of anti-cyclic operads.
\end{lem}
\begin{proof}
We use the model structures on reduced operads and reduced cyclic operads. Note that the map $p_\O$ is an acyclic fibration and therefore so is the map $p_\O\otimes\id$. Since the operad $\B\Com$ is cofibrant the lifting $\Phi_{\O}$ must exist and is unique up to homotopy by standard model structure arguments.
\end{proof}
The above lemma allows us to construct an $L_\infty$~algebra structure, well defined up to homotopy, on a tensor product of a $\B\Lambda \O^!$--algebra and a $\B\O$--algebra, extending the case of a tensor product of an $\O$--algebra and a $\B\O$--algebra. Indeed, if $V$ is a $\B\O$--algebra corresponding to an operad map $f\co\B\O\to\E(V)$ and $A$ is a $\B\Lambda\O^!$--algebra corresponding to an operad map $g\co\B\Lambda\O^!\to\E(A)$ then there is a map \[(g\otimes f)\circ\Phi_{\O}\co\B\Com\to\E(A)\otimes\E(V)\to
\E(A\otimes V)\] specifying an $L_\infty$~structure on $A\otimes V$. This construction depends on the choice of $\Phi_{\O}$ but different choices lead to homotopic operad maps and therefore $L_\infty$~isomorphic $L_{\infty}$~algebra structures on $A\otimes V$.

Similarly, let $\O$ be a cyclic operad, $V$ be a space with a weakly non-degenerate odd bilinear form and $f\co \B\O\to \Ecyc(V)$ be a $\B\O$--algebra structure on $V$. Then given a space $A$ with a weakly non-degenerate even bilinear form and $g\co\B\Lambda\O^!\to \Ecyc(A)$ a $\B\Lambda\O^!$--algebra structure on $A$, there is a map \[(g\otimes f)\circ\Phi_\O\co \B\Com \to \Ecyc(A)\otimes \Ecyc(V) \to \Ecyc(A\otimes V)\] in other words an odd cyclic $L_\infty$~structure on $A\otimes V$, well defined up to homotopy of anti-cyclic operad maps (and hence well defined up to cyclic $L_\infty$~isomorphism). If $\O$ is anti-cyclic, then $\Lambda\O^!$ is cyclic and the same construction applies.

We will now show that our construction of a tensor product is homotopy invariant in each argument.
\begin{theorem}\label{thm:homotopyequiv}
Let $\O$ be a Koszul operad and $A,A^\prime$ be homotopy equivalent $\B\Lambda\O^!$--algebras and $V$ be an $\B\O$--algebra. Then the $L_\infty$~algebras $A\otimes V$ and $A^{\prime}\otimes V$ are homotopy equivalent.

In the case $\O$ is a cyclic or anti-cyclic Koszul operad then the \emph{odd cyclic} $L_\infty$~algebras $A\otimes V$ and $A'\otimes V$ are homotopy equivalent.
\end{theorem}
\begin{proof}
Since $A$ and $A^\prime$ are homotopy equivalent, there is a common model $B$ for both, so the following diagram of operads is commutative up to homotopy.
\[
\xymatrix
{
&\B\Lambda\O^!\ar[dl]\ar[d]\ar[dr]\\
\E(A)&\E(B)\ar@{_{(}->}[l]\ar@{^{(}->}[r]&\E(A^\prime)
}
\]
Tensoring $A,A^\prime$ and $B$ with $V$, we obtain a homotopy commutative diagram
\[
\xymatrix
{
&\B\Com\ar^{\Phi_\O}[d]\\
&\B\Lambda\O^!\otimes\B\O\ar[dl]\ar[d]\ar[dr]\\
\E(A\otimes V)&\E(B\otimes V)\ar@{_{(}->}[l]\ar@{^{(}->}[r]&\E(A^\prime\otimes V)
}
\]
Thus, the $L_\infty$~algebras $A\otimes V$ and $A^{\prime}\otimes V$ are homotopy equivalent.

The case when $\O$ is cyclic or anti-cyclic is proved in the same way, by replacing the endomorphism operads with cyclic or anti-cyclic endomorphism operads, as appropriate.
\end{proof}

\end{document}